\documentclass[12pt,a4paper,english,fleqn]{article}

\pdfoutput=1

\usepackage[english]{babel}
\RequirePackage{ucs}
\RequirePackage[utf8x]{inputenc}
\RequirePackage[T1]{fontenc}
\RequirePackage{amsthm,amsmath}
\RequirePackage{amssymb,amstext}
\RequirePackage{amsfonts,stmaryrd}
\RequirePackage{amsbsy} 
\RequirePackage{mathrsfs}
\RequirePackage{bm}
\RequirePackage{bbm}
\RequirePackage{aliascnt,ifthen}
\RequirePackage{graphicx}
\RequirePackage{color,calc}
\RequirePackage{times}
\RequirePackage{booktabs}
\RequirePackage{dsfont} 
\usepackage{url}
\usepackage{verbatim}
\usepackage{subcaption}
\usepackage{tabularx}
\usepackage[pdfpagemode=UseOutlines ,plainpages=false,
hyperindex=true,colorlinks=true]{hyperref}
	\makeatletter
	\Hy@breaklinkstrue
	\makeatother

\definecolor{darkred}{rgb}{0.6,0,0.1}
\definecolor{darkgreen}{rgb}{0,0.5,0}
\definecolor{darkblue}{rgb}{0,0,0.5}

\hypersetup{colorlinks ,linkcolor=gray ,filecolor=darkgreen, urlcolor=darkblue ,citecolor=black}

\RequirePackage{paralist}
\RequirePackage{relsize}
\usepackage{enumitem}

\usepackage{natbib}
\usepackage{hypernat}
\renewcommand{\cite}{\citet}
\usepackage{abrev-package}

\definecolor{dgreen}{rgb}{0,0.5,0}
\definecolor{dblue}{rgb}{0,0,0.5}
\definecolor{dred}{rgb}{0.6,0.0,0.1}
\definecolor{dgold}{rgb}{0.5,0.3,0.0}
\definecolor{dvio}{rgb}{0.6,0.3,0.5}
\definecolor{gray}{rgb}{0.5,0.5,0.5}
\definecolor{dbraun}{rgb}{.5,0.2,0}

\newcommand{\dgrau}{\color{gray}}

\newcommand{\colre}{dred}
\newcommand{\colas}{dblue}
\newcommand{\colrem}{dgold}
\newcommand{\colil}{dgreen}

\newtheoremstyle{styre}
  {1.1\topsep}
  {\topsep}
  {\normalfont\itshape}
  {}
  {\color{\colre}}
  {.}
  {.5em}
  {\thmname{\textbf{#1}\xspace}{\xspace\dgrau\thmnumber{#2}}\thmnote{\xspace\textit{\small(#3)}}}

\newtheoremstyle{styas}
  {1.1\topsep}
  {\topsep}
  {\normalfont\itshape}
  {}
  {\color{\colas}}
  {.}
  {.5em}
  {\thmname{\textbf{#1}\xspace}{\xspace\dgrau\thmnumber{#2}}\thmnote{\xspace\textit{\small(#3)}}}

\newtheoremstyle{styrem}
  {1.1\topsep}
  {\topsep}
  {\normalfont\itshape}
  {}
  {\color{\colrem}}
  {.}
  {.5em}
  {\thmname{\textbf{#1}\xspace}{\xspace\dgrau\thmnumber{#2}}\thmnote{\xspace\textit{\small(#3)}}}

\newtheoremstyle{styil}
  {1.1\topsep}
  {\topsep}
  {\normalfont\rmfamily}
  {}
  {\color{\colil}}
  {.}
  {.5em}
  {\thmname{\textbf{#1}\xspace}{\xspace\dgrau\thmnumber{#2}}\thmnote{\xspace\textit{\small(#3)}}}

\newtheoremstyle{stypro}%
	{0.5\topsep}
	{1.1\topsep}
	{\upshape}
	{}
	{}
	{.}
	{.5em}
	{\thmnote{\textit{#3}}}

\theoremstyle{styre}\newtheorem{pr}{Proposition}[section]
\newaliascnt{co}{pr}
\theoremstyle{styre}\newtheorem{co}[co]{Corollary}
\aliascntresetthe{co}
\newaliascnt{thm}{pr}
\theoremstyle{styre}\newtheorem{thm}[thm]{Theorem}
\aliascntresetthe{thm}
\newaliascnt{lem}{pr}
\theoremstyle{styre}\newtheorem{lem}[lem]{Lemma}
\aliascntresetthe{lem}
\newaliascnt{rem}{pr}
\theoremstyle{styrem}\newtheorem{rem}[rem]{Remark}
\aliascntresetthe{rem}
\newaliascnt{il}{pr}
\theoremstyle{styil}\newtheorem{il}[il]{Illustration}
\aliascntresetthe{il}
\theoremstyle{styas}\newtheorem{ass}{Assumption}

\theoremstyle{styas}

\theoremstyle{stypro}

\usepackage{cleveref}
\crefname{pr}{\color{\colre}Proposition}{\color{\colre}Propositions}
\crefname{co}{\color{\colre}Corollary}{\color{\colre}Corollaries}
\crefname{thm}{\color{\colre}Theorem}{\color{\colre}Theorems}
\crefname{lem}{\color{\colre}Lemma}{\color{\colre}Lemmata}
\crefname{ass}{\color{\colas}Assumption}{\color{\colas}Assumptions}
\crefname{de}{\color{\colas}Definition}{\color{\colas}Definitions}
\crefname{rem}{\color{\colrem}Remark}{\color{\colrm}Remarks}
\crefname{il}{\color{\colil}Illustration}{\color{\colil}Illustrations}

\usepackage{environ}
\NewEnviron{te}{%
{\BODY}%
}

\numberwithin{equation}{section}

\newcommand{\setListe}[5][3ex]{\setlength{\itemsep}{#2}\setlength{\topsep}{#3}\setlength{\leftmargin}{#4}\setlength{\rightmargin}{#5}\setlength{\labelwidth}{#1}}
\newcommand{\setListeStandard}{\setListe{0ex}{.5ex}{4ex}{0ex}}
\newcounter{ListeN}
\renewcommand{\theListeN}{(\roman{ListeN})}

{\setcounter{ListeN}{0}\renewcommand{\theListeN}{\normalfont\rmfamily\color{\colre}(\roman{ListeN})}\begin{list}{\theListeN}%
{\usecounter{ListeN}\setListeStandard #1}}
{\end{list}}

{\begin{list}{}%
{\setListeStandard #1}}
{\end{list}}

\makeatletter%
\def\@fnsymbol#1{\ensuremath{\ifcase#1\or * \or ** \or 2 \or 3 \or  *\or  \star \or 4\or  , \or 
g\or h\or i\else\@ctrerr\fi}}%
\makeatother%

\author{\begin{minipage}{.45\textwidth}\center{\sc Bianca Neubert}\;\thanks{Institut f\"ur Mathematik, M$\Lambda$THEM$\Lambda$TIKON, Im Neuenheimer Feld 205,
			D-69120 Heidelberg, Germany, e-mail:
			\url{{neubert|johannes}@math.uni-heidelberg.de}}\\\small Ruprecht-Karls-Universität Heidelberg\\\null
	\end{minipage} \and \begin{minipage}{.45\textwidth}\center{\sc 
			Fabienne Comte} \thanks{Universit\'e Paris Cité, MAP5 UMR 8145,
			F-75006 Paris, France, e-mail:
			\url{fabienne.comte@u-paris.fr}}\\\small Universit\'e Paris Cité\\\null\end{minipage}\and\begin{minipage}{.45\textwidth}\center{\sc 
			Jan Johannes}$\;^*$\\\small Ruprecht-Karls-Universität Heidelberg\\\null\end{minipage}}
\date{} 
\title{Quadratic functional estimation from observations with multiplicative measurement error} 
\begin{document} 
\maketitle 
\begin{abstract}
We consider the nonparametric estimation of the value of a quadratic functional evaluated at the density of a strictly positive random variable $X$ based on an
  iid. sample from an observation $Y$ of $X$ corrupted by an independent multiplicative error $U$. Quadratic functionals of the density
  covered are the $\Lp[2]{}$-norm of the density and its derivatives or 
  the survival function. We construct a
  fully data-driven estimator when the error density is
  known. The plug-in estimator is based on a density estimation combining the estimation of the Mellin transform of the $Y$ density 
  and a spectral cut-off regularized inversion of the Mellin transform of the
  error density. The main issue is the data-driven choice of the
  cut-off parameter using a Goldenshluger-Lepski-method. We
  discuss conditions under which the fully data-driven estimator attains oracle-rates up to logarithmic deteriorations. We compute convergence rates under classical
  smoothness assumptions and illustrate them by a simulation study.
\end{abstract} 
{\footnotesize
\begin{tabbing} 
\>Multiplicative measurement errors, Quadratic functional\\[.2ex] 
\end{tabbing}}

\section{Introduction}
We consider the nonparametric estimation of the value of a quadratic functional evaluated at the density of a strictly positive random variable $X$ in a multiplicative measurement error model. More precisely, we have access to an independent and identically distributed (iid.) sample of size $n\in\Nz$ from $Y = X U$, where $X$ and $U$ are strictly positive independent random variables.  We denote the unknown density of $X$  by $\denX$ and assume that $U$ admits a known density $\denU$. We refer to the density of $Y$ by $\denY$, which is then given by the multiplicative convolution of $\denX$ and $\denU$, i.e.
 \begin{align}\label{eq::multiplicative-convoluton}
 	\denY (y) = (\denX * \denU)(y):= \int_{\pRz}  \denX(x) \denU(y/x)x^{-1}dx.
 \end{align}
 Consequently, making inference based on observations of $Y$ is a deconvolution problem which can also be seen as a statistical inverse problem. 

 One motivation for estimating quadratic functionals comes from testing theory: A quadratic functional of the difference between unknown density and null hypothesis provides an intuitive test statistic. It should also be mentioned that the test problem yields a lower bound for the estimation problem. For the additive error model, \cite{Butucea2007} and \cite{SchluttenhoferJohannes2020b,SchluttenhoferJohannes2020} discuss extensively the link between testing theory and quadratic functional estimation. Quadratic functional estimation has received great attention in the past. In the case of direct observations, \cite{bickel1988} are the first to discover a typical phenomenon in quadratic functional estimation in the density context: the so-called elbow effect. It refers to a sudden change in the behavior of the rates, as soon as the smoothness parameter crosses a critical threshold.  \cite{Birge1995}  look at minimax lower bounds concerning smooth functionals of the density. \cite{DONOHO1990290} discover a similar behavior in the Gaussian sequence space model, the smoothness assumptions being in this context replaced by geometric assumptions. \cite{Laurent96} shows minimax rates, where also more details on the density framework and simpler estimators can be found. Adaptive nonparametric quadratic functional estimation is considered in a Gaussian sequence model by  \cite{LaurentMassart2000} and in a density context by \cite{Laurent2005}. In the case of indirect observations, quadratic functional estimation in an inverse Gaussian sequence space model is treated by \cite{BUTUCEA201131} for a known operator and in \cite{Kroll2019} for partially unknown operators.
For the additive measurement error model on the real line, quadratic functional estimation is considered, for example, by \cite{Chesneau2011}, \cite{Meister2009} and \cite{Butucea2007}.
Circular observations are studied in \cite{SchluttenhoferJohannes2020b,SchluttenhoferJohannes2020}. 

In this work we extend quadratic functional estimation to the multiplicative measurement error model. Multiplicative censoring, which corresponds to the multiplicative measurement error model with $U$ being uniformly distributed on $[0,1]$,  has been introduced and studied by \cite{Vardi1989} and \cite{VardiZhang1992}. \cite{VANES2000295} explain and motivate the use of multiplicative censoring models in survival analysis. Concerning inference in this model, there has been extensive work on estimation of the unknown density $\denX$.
To name a few: \cite{Andersen2001} consider series expansion methods treating the model as an inverse problem, \cite{brunel2016} use a kernel estimator for density estimation in the multiplicative censoring model, \cite{Comte2016} consider a projection density estimator with respect to the Laguerre basis, \cite{Belomestny2016} study a Beta-distributed error. The multiplicative measurement error model covers all of these cases of multiplicative censoring. Nonparametric density estimation in the multiplicative measurement error model has been considered by \cite{Brenner2021} using a spectral cut-off regularization.  \cite{BrennerMiguel2022}  considers an estimation procedure using an anisotropic spectral cut-off.  \cite{Miguel2021} look at the estimation of a linear functional of the unknown density and  \cite{Belomnestny2020} examine pointwise density estimation in the multiplicative measurement error model.

We now turn to the nonparametric estimation of the value of a weighted quadratic functional evaluated at the Mellin transform of the unknown density in a multiplicative measurement error model. Thereby, our work extends the results of the estimation of the value of a quadratic functional evaluated at the density to the multiplicative measurement error model and, further, its generalized formulation covers not only the estimation the (possibly weighted) $\Lp[2]{}$-norm of the density but also the $\Lp[2]{}$-norm of its derivatives or the $\Lp[2]{}$-norm of its survival function. More precisely, by applying a Plancherel type identity those examples can be written as the value of a weighted quadratic functional evaluated at the Mellin transform of the unknown density. Following the estimation strategy in \cite{Brenner2021} we define a spectral cut-off estimator of the unknown density $\denX$ which we plug-in the quadratic functional. The accuracy of the estimator is measured by its mean squared error and it depends crucially on the cut-off parameter. Our aim is to establish a fully data-driven estimation procedure inspired by \cite{GL2011} and to derive upper bounds for its mean squared error as well as convergence rates.

The paper is organized as follows. We start in \cref{sec::Methodology} by introducing the multiplicative measurement error model, giving a brief introduction to the Mellin transform and presenting the estimation problem together with the estimation strategy. In \cref{sec::Upper-bound} we first decompose the estimation error of the plug-in estimator appropriately and then derive an upper bound of its mean squared error. Introducing Mellin-Sobolev spaces we derive convergence rates resulting from the upper bound. Furthermore, we introduce a data-driven procedure based on the Goldenshluger-Lepski method, show an upper bound for its mean squared error and calculate convergence rates in \cref{sec::Adaptive}. Finally, we illustrate our results by a short numerical study in \cref{sec::Numerical-Study}. Proofs can be found in the appendix.
\section{Methodology}\label{sec::Methodology}

In this section we introduce the multiplicative measurement error model. Then, we give a short overview of the Mellin transform and recall certain properties. With this we state the estimation problem, that is, define the quadratic functional of our interest. Finally, we outline the estimation strategy and gather assumptions made throughout this work.

\paragraph{The multiplicative measurement error model}
Assume that $X$ and $U$ are independent random variables taking strictly positive real values and admitting a Lebesgue density $\denX$ and $\denU$, respectively. The multiplicative error measurement model describes an observation
\begin{align*}
	Y= X\cdot U.
\end{align*}
In this situation, $Y$ admits a Lebesgue density $\denY = \denX * \denU$ given as the multiplicative convolution of $\denX$ and $\denU$, see  \cref{eq::multiplicative-convoluton}. For a detailed discussion of multiplicative convolution and its properties as an operator between function spaces we refer to \cite{BrennerMiguelDiss}. Assuming throughout this work that the error density $\denU$ is known, we have access to an iid. sample $(Y_j)_{j\in\llbracket n\rrbracket}$ drawn from $Y$, where we have used the shorthand notation $\llbracket a \rrbracket := [1,n] \cap \Nz$ for any $n\in\Nz$.

\paragraph{The Mellin transform} 
In the subsequent, we keep the following objects and notations in mind: Let $(\pRz, \psB, \pLm)$ denote the Lebesgue-measure space of all strictly positive real numbers $\pRz$ equipped with the restriction $\pLm$ of the Lebesgue measure on the Borel $\sigma$-field $\psB$. In contrast, denote by $\Rz_{\geq 0}$ the positive real line.  Given a density function $\nu$ defined on $\pRz$, i.e. a Borel-measurable nonnegative function $\nu\colon\pRz \rightarrow \Rz_{\geq 0}$, let $\nu\pLm$ denote the $\sigma$-finite measure on $(\pRz, \psB)$ which is $\pLm$ absolutely continuous and admits the Radon-Nikodym derivative $\nu$ with respect to $\pLm$. For $p \in [1, \infty]$  let $\Lp[p]{+}(\nu) := \Lp[p]{+}(\pRz, \psB, \nu\pLm)$ denote the usual complex Banach-space of all (equivalence classes of) $\Lp[p]{+}$-integrable complex-valued function with respect to the measure $\nu\pLm$. Also, we set $\Lp[p]{}(\nu):= \Lp[p]{}(\Rz, \sB, \nu\Lm)$ for a density function $\nu$ defined on $\Rz$. 
For $p\in\pRz$ we define the weighted $\Lp[p]{+}(\nu)$-norm of any measurable complex-valued function $h\colon\pRz\rightarrow\Cz$ by the term  $\Vert h\Vert^p_{\Lp[p]{+}(\nu)}:= \int_{\pRz} \vert h(x)\vert^p \nu(x)dx$ and denote by $\Vert h\Vert_{\Lp[\infty]{+}(\nu)}$ the essential supremum of the function $x\mapsto h(x)\nu(x)$. For $p=2$ and $h_1,h_2\in\Lp[2]{+}(\nu)$ let $\langle h_1, h_2\rangle_{\Lp[2]{+}(\nu)}:= \int_{\pRz} h_1(x)\overline{h_2(x)} \nu(x)dx$. Analogously, define the $\Lp[p]{}(\nu)$-norm and $\Lp[2]{}(\nu)$-scalar product.
If $\nu = 1$, i.e. $\nu$ is mapping constantly to one, we write $\Lp[p]{+}:= \Lp[p]{+}(1)$ and $\Lp[p]{} := \Lp[p]{}(1)$.
At this point we shall remark, that we have used and will further use the terminology density, whenever we are meaning a probability density function (such as $\denX$) and on the other hand side density function, whenever we are meaning a Radon-Nikodym derivative (such as $\nu$). For $c\in\Rz$ we introduce the density function $\basMSy{c}\colon\pRz\rightarrow \pRz$ given by $x\mapsto \basMSy{c}(x):= x^c$. The Mellin transform $\mathcal{M}_c$ is the unique linear and bounded operator between the Hilbert spaces $\Lp[2]{+}(\basMSy{2c-1})$ and $\Lp[2]{}$, which for each $h\in\Lp[1]{+}(\basMSy{c-1})\cap \Lp[2]{+}(\basMSy{2c-1})$  and $t\in\Rz$ satisfies
\begin{align*}
	\Mcg{c}{h}(t)= \int_{\pRz} h(x) x^{c-1+ 2\pi it}dx,
\end{align*}
where $i\in\mathbb{C}$ denotes the imaginary unit. Similar to the Fourier transform, the Mellin transform $\mathcal{M}_c$ is unitary, i.e. $	\langle h_1, h_2 \rangle_{\Lp[2]{+}(\basMSy{2c-1})} = \langle \Mcg{c}{h_1}, \Mcg{c}{h_2} \rangle_{\Lp[2]{}}$ for any $h_1,h_2 \in \Lp[2]{+}(\basMSy{2c-1})$. In particular, it satisfies a \textit{Plancherel type identity}, i.e. for all $ h\in \Lp[2]{+}(\basMSy{2c-1})$
\begin{align}\label{eq::Plancherel}
	\Vert h \Vert^2_{\Lp[2]{+}(\basMSy{2c-1})} = \Vert \Mcg{c}{h} \Vert_{\Lp[2]{}}^2.
\end{align}
Its adjoined and inverse  $\Mgi{c}\colon \Lp[2]{} \rightarrow \Lp[2]{+}(\basMSy{2c-1})$ fulfills for each $H\in\Lp[1]{}\cap \Lp[2]{}$ and $x\in\pRz$ that
\begin{align*}
	\Mcgi{c}{H} (x) =\int_\Rz x^{-c- 2\pi i t} H(t) dt.
\end{align*}
In analogy to the additive convolution theorem of the Fourier transform (see \cite{Meister2009} for definitions and properties), there is a multiplicative convolution theorem for the Mellin transform. That is, for any $c\in\Rz$, $h_1 \in \Lp[2]{+}(\basMSy{2c-1})$ and $h_2 \in \Lp[1]{+}(\basMSy{c-1})\cap\Lp[2]{+}(\basMSy{2c-1})$ the \textit{multiplicative convolution theorem} states
\begin{align}\label{eq::convolution-theorem}
	\Mcg{c}{h_1 * h_2} = \Mcg{c}{h_1} \cdot \Mcg{c}{h_2}.
\end{align}
For a detailed discussion of the Mellin transform and its properties we refer again to \cite{BrennerMiguelDiss}. 

\paragraph{Estimation problem}
 Given a fixed $c\in\Rz$ and an arbitrary measurable symmetric density function $\wFSy\colon\Rz \rightarrow \Rz_{\geq 0}$ we aim to estimate the value of the quadratic functional 
\begin{align}\label{eq::parameter}
	\wqf := \wqf(\denX) := \Vert \Mcg{c}{\denX} \Vert_{\Lp[2]{}(\iwF[]{2})}^2 = \int_{\Rz}  \vert \Mcg{c}{\denX} (t)\vert^2 \iwF[]{2} (t) dt
\end{align}
in the multiplicative measurement error model.

\begin{il}\label{il::example-parameters}
	 This general estimation problem includes the following possible quadratic functionals of interest.
	\begin{itemize}
		\item[(i)] If $\wFSy = 1$, we get using the Plancherel type identity \cref{eq::Plancherel} that
		\begin{align*}
			\wqf(\denX) = \Vert \Mcg{c}{\denX} \Vert_{\Lp[2]{}}^2 = \Vert \denX \Vert_{\Lp[2]{+}(\basMSy{2c-1})}^2 = \int_{\pRz} \vert f(x)\vert^2 x^{2c-1} dx.
		\end{align*}
		Consequently, we cover the estimation of the value of a quadratic functional evaluated at the density $\denX$ itself. Note that $\wqf = \Vert \denX \Vert_{\Lp[2]{+}}^2$ in the special case of $c=\frac12$.
		\item[(ii)] If $\iwF{2}(t)= \frac{1}{(c-1)^2+4\pi^2 t^2}$ for all $t\in\Rz$, the parameter of interest equals the  value of a quadratic functional evaluated at the survival function  $S$ of $X$, i.e.
		\begin{align*}
			\wqf (f)= \Vert S \Vert_{\Lp[2]{+}(\basMSy{2c-1})} =  \int_{\pRz} \vert S(x)\vert^2 x^{2c-1} dx.
		\end{align*}
		In the special case of $c=\frac12$  the $\Lp[2]{+}$-norm $\Vert S \Vert_{\Lp[2]{+}}^2$ of the survival function $S$ is considered. For more details see \cite{BrennerMiguel-Phandoidaen2023}.
		\item[(iii)] Let $\denX\in\mathcal{C}^\infty_0(\pRz)$, i.e. smooth with compact support in $\pRz$ such that for any $\beta\in\Nz$ the derivatives $D^\beta[\denX]:= \frac{d^\beta}{dx^\beta}f$ exist. 
		If $\iwF{2}(t)= \prod_{j=1}^\beta ((c+\beta-j)^2+4\pi^2t^2)$ for  $\beta\in\Nz$ then we have that 
		\begin{align*}
			\wqf (f)= \Vert D^\beta[\denX] \Vert_{\Lp[2]{+}(\basMSy{2(c+\beta)-1})} =  \int_{\pRz} \vert D^\beta[\denX](t)\vert^2 x^{2(c+\beta)-1} dx,
		\end{align*}
		see Proposition 2.3.11, \cite{BrennerMiguelDiss}. In the special case $c=\frac12 - \beta$ the $\Lp[2]{+}$-norm $\Vert D^\beta[\denX] \Vert_{\Lp[2]{+}}^2$ of the derivative $D^\beta[\denX]$ is considered.
		\end{itemize}
\end{il}
For any set $A \subseteq \Rz$ we write $\mathds{1}_A$ for the indicator function, mapping $x\in\Rz$ to $\mathds{1}_A(x) := 1$, if $x\in A$, and $\mathds{1}_A(x) := 0$, otherwise. In the following, for a random variable $Y$ that admits a Lebesgue density $g$ denote by $\pM[g]$ and $\Ex[g]$ the corresponding probability measure and the expectation, respectively. For $(Y_j)_{j\in\llbracket n\rrbracket}$ iid. denote the corresponding product measure and expectation by $\ipM[g]{n}$ and $\iEx[g]{n}$ respectively. 

\paragraph{Estimation strategy}
We intend to plug-in an estimator of the Mellin transform of the unknown density $\denX$. We follow the approach of \cite{Brenner2021} for density estimation in the multiplicative measurement error model. We first motivate the estimation, necessary assumptions are summarized in \cref{ass:well-definedness} below. Using the multiplicative convolution theorem  \cref{eq::convolution-theorem}, we have that $\Mcg{c}{\denX} = \Mcg{c}{\denY}/\Mcg{c}{\denU}$, 
so to get an estimator of $\wqf$ one would be tempted to plug in  \cref{eq::parameter} an estimator for the unknown Mellin transform $\Mcg{c}{\denY}$, i.e. the empirical Mellin transform
\begin{align*}
	\widehat{\mathcal{M}}_c(t) :=  \frac1n \sum_{j\in\llbracket n\rrbracket} Y_j^{c-1+2\pi it}.
\end{align*} 
Observe that $\widehat{\mathcal{M}}_c$ is a pointwise unbiased estimator of $\Mcg{c}{\denY}$, i.e.
\begin{align*}
	\iEx[\denY]{n} [\widehat{\mathcal{M}}_c(t)] = \Ex[\denY][Y_1^{c-1+2\pi it}] = \Mcg{c}{\denY}(t).
\end{align*}
Also it holds for the pointwise scaled variance of $\widehat{\mathcal{M}}_c$ that for each $t\in\Rz$
\begin{align}\label{eq::emt-variance}
	n \iEx[\denY]{n} [\vert \widehat{\mathcal{M}}_c(t)- \Mcg{c}{\denY}(t)\vert^2] &=  \Ex[\denY] [\vert Y_1^{c-1+2\pi it}- \Mcg{c}{\denY}(t)\vert^2]\nonumber\\
	& \leq  \Ex[\denY] [Y_1^{2c-2}] = \Vert \denY \Vert_{\Lp[1]{+}(\basMSy{2c-2})}.
\end{align}
However, $\widehat{\mathcal{M}}_c/\Mcg{c}{\denU}$ is in general not square integrable.
A frequently used technique in inverse problems is the so called spectral cut-off method, see for example \cite{Engl1996}. Applying this method to the estimator and plugging it into  \cref{eq::parameter} yields for some $k\in\pRz$ the term
\begin{align*}
	\left\Vert \mathds{1}_{[-k,k]}  \frac{\widehat{\mathcal{M}}_c}{\Mcg{c}{\denU}}  \right\Vert_{\Lp[2]{}(\iwF{2})}^2  .
\end{align*}
Correcting for the bias yields the estimator
\begin{align}
	\eqfk{k} :=   \frac{1}{n(n-1)}\sum_{\substack{j\not= l\\ j,l \in\llbracket n\rrbracket }}  \int_{-k}^k \frac{Y_{j}^{c-1+2\pi it}Y_l^{c-1-2\pi it}}{\vert\Mcg{c}{\denU}(t)\vert^2} \iwF[]{2} (t)  dt. \label{eq::quadraticestim}
\end{align}
Using the independence of $(Y_j)_{j\in\llbracket n\rrbracket }$ we have that
\begin{align*}
	\wqfk{k} &:= \iEx[\denY]{n} [\eqfk{k}] = \int_{-k}^k \frac{\Ex[\denY][Y_{j}^{c-1+2\pi it}] \Ex[\denY][Y_l^{c-1-2\pi it}]}{\vert\Mcg{c}{\denU}(t)\vert^2} \iwF[]{2} (t)  dt \\
	&= \int_{-k}^k  \vert \Mcg{c}{\denX}(t)\vert^2  \iwF[]{2} (t)  dt
\end{align*}
All needed properties for the estimation problem and the estimator $\eqfk{k}$ to be well-defined are summarized in the following assumption.
\begin{ass}\label{ass:well-definedness}
	Consider the multiplicative measurement error model, an arbitrary measurable symmetric density function $\wFSy\colon\Rz \rightarrow \Rz_{\geq 0}$ and $c\in\Rz$. In addition, let
	\begin{itemize}
		\item[(i)] $\denX, \denU  \in\Lp[2]{+}(\basMSy{2c-1})$,
		\item[(ii)] $\Mcg{c}{\denX}\in\Lp[2]{}(\iwF{2})$,
		\item[(iii)] $\Mcg{c}{\denU}(t)\not=0$ for all $t\in\Rz$ and $\mathds{1}_{[-k,k]}/\Mcg{c}{\denU} \in\Lp[2]{}(\iwF{2})$,
		\item[(iv)] $\denX,\denU\in\Lp[1]{+}(\basMSy{2c-2})$ and $\denU \in\Lp[\infty]{+}(\basMSy{2c-1})$.
	\end{itemize}
\end{ass}
\begin{rem}  With \cref{ass:well-definedness} (i) the existence of the Mellin transform of $\denX$ and $\denU$ is ensured. Note that since $\denY = \denX * \denU$  it follows $\denY \in \Lp[2]{+}(\basMSy{2c-1})$, see \cref{lem::normineq} in the appendix.
	Under \cref{ass:well-definedness} (ii) the well-definedness of the parameter $\wqf$ of interest, see  \cref{eq::parameter}, is guaranteed. With additionally \cref{ass:well-definedness} (iii) the estimator $\eqfk{k}$ in  \cref{eq::quadraticestim} is well-defined. It also ensures that $\mathds{1}_{[-k,k]}\in\Lp[2]{}(\iwF{2})$, since $\mathds{1}_{[-k,k]}\leq  \Ex[\denU][U^{c-1}]\mathds{1}_{[-k,k]}/\Mcg{c}{\denU} $. Under \cref{ass:well-definedness} (iv), it follows that $\denY\in\Lp[1]{+}(\basMSy{2c-2})\cap\Lp[\infty]{+}(\basMSy{2c-1})$ and $\denU\in\Lp[1]{+}(\basMSy{c-1})$, hence, with (i) we can apply the multiplicative convolution theorem  \cref{eq::convolution-theorem}.  Note that $\Lp[2]{+}$-integrability is a common assumption in additive deconvolution, which might be here seen as the particular case $c=\frac12$. However, allowing for different values $c\in\Rz$ makes the dependence on the assumptions visible.
\end{rem}

\section{Upper bound and rates}\label{sec::Upper-bound}
In this chapter, we first give an upper bound for the mean squared error of the spectral cut-off estimator $\eqfk{k}$ defined in  \cref{eq::quadraticestim} and secondly give convergence rates resulting from the upper bound.

\subsection{Upper bound}
In the following, we consider an upper bound on the mean squared error of the spectral cut-off estimator $\eqfk{k}$ .
For this goal, we first decompose the estimation error $\eqfk{k} - \wqf$. Introduce the canonical U-statistic $U_k$ given for $k\in\pRz$  by
\begin{align}
	U_k &:= \frac{1}{n(n-1)} \sum_{\substack{j\not= l\\ j,l \in\llbracket n\rrbracket }} \int_{-k}^k \frac{(Y_j^{c-1+2\pi it} -  \Mcg{c}{\denY}(t)) (Y_l^{c-1-2\pi it} -  \Mcg{c}{\denY}(-t))}{\vert \Mcg{c}{\denU}(t)\vert^2} \iwF[]{2} (t) dt \label{eq::u-stat}
	\intertext{and the centred linear statistic $W_k$ given by}
	W_k &:= \frac{1}{n} \sum_{j \in\llbracket n\rrbracket } \int_{-k}^k (Y_j^{c-1+2\pi it} -  \Mcg{c}{\denY}(t))\frac{ \Mcg{c}{\denX}(-t)}{ \Mcg{c}{\denU}(t)}  \iwF[]{2} (t)  dt. \label{eq::linear-term}
\end{align}
With this, we obtain for $\wqf$ defined in  \cref{eq::parameter} and $\eqfk{k}$ in  \cref{eq::quadraticestim} the decomposition 
\begin{align}\label{eq:decomp-1}
	\eqfk{k} - \wqf = U_k + 2W_k - \bias{k}{}.
\end{align}
Note that the centered statistics $U_k$ and $W_k$ are uncorrelated and we obtain
\begin{align}\label{eq::decomp}
	\iEx[\denY]{n}  \left[  \left\vert \eqfk{k}- \wqf \right\vert^2  \right] = 	\iEx[\denY]{n} [\vert U_k\vert^2] + 4 	\iEx[\denY]{n}  [\vert W_k\vert^2] + \bias{k}{2}.
\end{align}
Define for $k \in \pRz$
\begin{align}\label{eq::Delta-Lambda}
	\delfour{k} := \int_{-k}^k \frac{1}{\vert \Mcg{c}{\denU}(t) \vert^{4}} \iwF{4}(t) dt,  \quad \text{ and} \quad \Lambda_{\denX, \denU}(k):=  \int_{-k}^k \frac{\vert \Mcg{c}{\denX}(t)\vert^2}{ \vert \Mcg{c}{\denU}(t) \vert^{2} } \iwF{4}(t)dt.
\end{align} 
For $x,y\in\Rz$ denote $x\vee y := \max\{x, y\}$ and $x \wedge y := \min\{x,y\}$. Further define
\begin{align}\label{eq::constants}
    \cstV[\denX|\denU] := (\Vert \denY\Vert_{\Lp[1]{+}({\basMSy{2c-2}})} \vee 1) \quad \text{ and } \quad \cstC[\denU] := (\Vert \denU\Vert_{\Lp[\infty]{+}( \basMSy{2c-1})} / \Vert \denU\Vert_{\Lp[1]{+}(\basMSy{2c-2})}\vee 1),
\end{align}
which are finite under \cref{ass:well-definedness}; moreover, by  \cref{lem::normineq},  \cref{eq::mm3} it holds 
\begin{align}
	\Vert \denY\Vert_{\Lp[\infty]{+}(\basMSy{2c-1})}\leq \cstV[\denX|\denU]   \cstC[\denU], \quad \text{ and } \quad
	\Vert \denY  \Vert_{\Lp[\infty]{+}(\basMSy{2c-1})} \Vert \denY \Vert_{\Lp[1]{+}({\basMSy{2c-2}})} \leq  \cstC[\denU]  \icstV[\denX|\denU]{2} .   \label{eq::norm-bound}
\end{align}

\begin{lem}\label{lem::upperbound-u-lin}
	Under \cref{ass:well-definedness} we have for all $k\in\pRz$ and $n\in\Nz$, $n\geq 2$ that
	\begin{align*}
	\iEx[\denY]{n}  [\vert U_k\vert^2]  \leq 2 \cstC[\denU]  \icstV[\denX|\denU]{2}  \frac{\delfour{k} }{n^2}   \quad \text{ and } \quad \iEx[\denY]{n} [\vert W_k\vert^2]  \leq  \cstC[\denU] \cstV[\denX|\denU]  \frac{\Lambda_{\denX, \denU}(k)}{n}.
	\end{align*}
\end{lem}
\begin{proof}[Proof of \cref{lem::upperbound-u-lin}]
	We first consider $U_k$ defined in  \cref{eq::u-stat}. Define for $x,y\in\pRz$ the symmetric and real-valued function 
	\begin{align}\label{eq::kernel-ustat}
		h_k(x,y):=  \int_{-k}^{k} 				\frac{(x^{c-1+2\pi it}-\Mcg{c}{\denY}(t))(y^{c-1-2\pi it}-\Mcg{c}{\denY}(-t))}{\vert \Mcg{c}{\denU}(t)\vert^2} \iwF{2}(t) dt
	\end{align}
	Then, $U_k =  \frac{1}{n(n-1)} \sum_{\substack{j\not= l\\ j,l \in\llbracket n\rrbracket }} h_k(Y_j, Y_l)$, which is a canonical U-statistic, since it holds that $\Ex_\denY[h_k(y,Y_1)]= 0$ for all $y\in\pRz$.
	Using that $U$ is canonical by \cref{lem::u-stat-var} we get
	\begin{align*}
		\operatorname{Var}(U_k) \leq \frac{1}{n(n-1)}  \iEx[\denY]{2} [h^2_k(Y_1,Y_2)].
	\end{align*}
	For every $y\in\pRz$ using property  \cref{eq::mm1} of \cref{lem::Mellinprops} we have
	\begin{align*}
		\iEx[\denY]{} [h^2_k(Y_1,y)] &\leq \Ex[\denY] \left[  \left\vert  \int_{-k}^k Y_1^{c-1+2\pi it} \frac{(y^{c-1-2\pi it}-\Mcg{c}{\denY}(-t))}{\vert \Mcg{c}{\denU}(t) \vert^{2}}  \iwF{2}(t) dt\right\vert^2 \right]\\
		&\leq \Vert \denY  \Vert_{\Lp[\infty]{+}(\basMSy{2c-1})}  \int_{-k}^k \vert y^{c-1-2\pi it}-\Mcg{c}{\denY}(-t)\vert^2 \frac{\iwF{4}(t)}{\vert \Mcg{c}{\denU}(t) \vert^{4}}  dt.
	\end{align*}
	Hence, applying first  \cref{eq::emt-variance}  we have that
	\begin{align*}
		\iEx[\denY]{2} [h^2_k(Y_1,Y_2)] 
		&\leq \Vert \denY  \Vert_{\Lp[\infty]{+}(\basMSy{2c-1})}  \int_{-k}^k \Ex[\denY] \left[ \vert Y_1^{c-1-2\pi it}-\Mcg{c}{\denY}(-t)\vert^2\right] \frac{\iwF{4}(t)}{\vert \Mcg{c}{\denU}(t) \vert^{4}}  dt  \\
		&\leq \Vert \denY  \Vert_{\Lp[\infty]{+}(\basMSy{2c-1})} \Vert \denY \Vert_{\Lp[1]{+}({\basMSy{2c-2}})} \int_{-k}^k \frac{\iwF{4}(t)}{\vert \Mcg{c}{\denU}(t) \vert^{4}} dt
	\end{align*}
		and, with  \cref{eq::norm-bound}, for all $n\geq 2$
	\begin{align*}
		\iEx[\denY]{n} [\vert U_k\vert^2]  &\leq  \frac{1}{n(n-1)}   \cstC[\denU]  \icstV[\denX|\denU]{2}  \delfour{k}\leq \frac2{n^2} \cstC[\denU]  \icstV[\denX|\denU]{2} \delfour{k}.
	\end{align*}
	Considering $W_k$,   we get using  \cref{eq::mm1} of \cref{lem::Mellinprops}  that
	\begin{align*}
		n\iEx[\denY]{n} [\vert W_k\vert^2]&= n\operatorname{Var}(W_k) = \operatorname{Var} \left( \int_{-k}^k Y_1^{c-1+2\pi it}\frac{\Mcg{c}{\denX}(-t) }{ \Mcg{c}{\denU}(t)}\iwF{2}(t)  dt \right) \\
		&\leq  \Ex[\denY] \left[ \left\vert \int_{-k}^k Y_1^{c-1+2\pi it} \frac{\Mcg{c}{\denX}(-t)}{\Mcg{c}{\denU}(t)}  \iwF{2}(t)  dt \right\vert^2 \right]\\
		&\leq   \Vert \denY\Vert_{\Lp[\infty]{+}(\basMSy{2c-1})}\int_{-k}^k \frac{\vert \Mcg{c}{\denX}(t)\vert^2} {\vert\Mcg{c}{\denU}(t)\vert^2} \iwF{4}(t) dt ,
	\end{align*}
	which together with  \cref{eq::norm-bound} implies
	\begin{align*}
		\iEx[\denY]{n} [\vert W_k\vert^2] \leq  \cstC[\denU]  \cstV[\denX|\denU] \frac{\Lambda_{\denX, \denU}(k)}{n}.
	\end{align*}
	This completes the proof.
\end{proof}
From  \cref{eq::decomp} and \cref{lem::upperbound-u-lin} we get immediately the next result and omit the proof. 
\begin{pr}\label{prop::generalresult}
	Under \cref{ass:well-definedness} for  the estimator $\eqfk{k}$ defined in  \cref{eq::quadraticestim} from decomposition  \cref{eq::decomp} and \cref{lem::upperbound-u-lin} we get for any $k\in\pRz$
	\begin{align*}
		\iEx[\denY]{n} [\vert \eqfk{k}- \wqf \vert^2  ] &=  \bias{k}{2} +\iEx[\denY]{n}  [\vert U_k\vert^2]  + 4 \iEx[\denY]{n} [\vert W_k\vert^2]  \\
		&\leq \bias{k}{2}+ 2  \cstC[\denU]  \icstV[\denX|\denU]{2} \left(   \frac{\delfour{k}}{n^2} + \frac{\Lambda_{\denX, \denU}(k)}{n}\right).
	\end{align*}
\end{pr}

\begin{rem}
	Let us briefly discuss \cref{prop::generalresult}. First, note that $\delfour{k}$ and $\Lambda_{\denX, \denU}$ are monotonically increasing in $k$ while the bias term $\bias{k}{2}$ is monotonically decreasing in $k$. Consequently, for an optimal choice of $k$, one has to balance these terms, known as the classical bias-variance trade-off. In the case when both $\Lambda_{\denX, \denU}$ and $\delfour{k}$ are bounded, we retain a parametric rate. Note also that an optimal choice of $k$ depends on $n$ and the unknown density $\denX$. This motivates a data-driven choice, which we discuss below in \cref{sec::Adaptive} in more detail. Next, we discuss uniform boundedness of the upper bound of \cref{prop::generalresult} over a nonparametric class of functions.
\end{rem}	

\subsection{Rates}
From the upper bound result \cref{prop::generalresult}, we now derive rates for specific function classes for the three examples for parameter $\wqf$ in \cref{il::example-parameters}. 
\begin{ass}\label{ass::classes}
\begin{enumerate}
	\item[(i)] Let $\operatorname{s}\colon\Rz\rightarrow \Rz$ be a symmetric, monotonically increasing weight function such that $\operatorname{s}(t)\rightarrow \infty$ as $\vert t\vert \rightarrow \infty$. 
	\item[(ii)] Assume that the symmetric function $\wF /\operatorname{s}$ is monotonically non-increasing such that $\wF(t)/\operatorname{s}(t) = o(1)$ as $\vert t\vert \rightarrow \infty$.
\end{enumerate}
\end{ass}
For $\operatorname{s}$ satisfying \cref{ass::classes} and $L>0$ we set
\begin{align}\label{eq::function-class-general}
	\mathcal{F}(\operatorname{s},L) := \{ &f\in \Lp[2]{+}(\basMSy{2c-1}) : \Vert \operatorname{s} \Mcg{c}{f} \Vert^2_{\Lp[2]{}} \leq L^2, \Vert f \Vert_{\Lp[1]{+}(\basMSy{2c-2})}\leq L^2\}.
\end{align}

\begin{rem}
	\cref{ass::classes} (i) covers the usual assumptions on the regularity of the unknown density $\denX$, e.g. ordinary and super smooth densities, see \cref{il::rates-non-adaptive} below. Concerning \cref{ass::classes} (ii), boundedness of $\wF / \operatorname{s}$ ensures that the quadratic functional is finite on the nonparametric class of functions $\mathcal{F}(\operatorname{s},L) $. Convergence against zero gives the rate for the bias term, as can be seen in the next result \cref{co::rates}. 
\end{rem}
For the nonparametric class $\mathcal{F}(\operatorname{s},L) $ of functions we derive a uniform bound for the risk. For this define for $k\in\pRz$ 
\begin{align}\label{eq::delta-inf}
	\delinf{k} := \left\Vert \frac{\wF\mathds{1}_{[-k,k]}}{\Mcg{c}{\denU}} \right\Vert_{\Lp[\infty]{}}^4.
\end{align}
Note that $\delfour{k}\leq (2k) \delinf{k}$ for $\delfour{k}$ defined in  \cref{eq::Delta-Lambda}. Further, let
\begin{align}\label{eq::rates-elbow-dim}
	R_n(k) :=  {\iwF{4}(k)}{\operatorname{s}^4(k)} \vee \frac{\delinf{k} \vee \delfour{k}}{n^2}  \text{ and } R^{\operatorname{elbow}}_n := \left\Vert \frac{\iwF{4}}{\operatorname{s}^4} \wedge \frac{\iwF{4}}{\operatorname{s}^2 n \vert\Mcg{c}{\denU}\vert^2} \right\Vert_{\Lp[\infty]{}}.
\end{align}

\begin{co}\label{co::rates}
	Under \cref{ass:well-definedness} and \cref{ass::classes}, if $f\in\mathcal{F}(\operatorname{s},L)$, see  \cref{eq::function-class-general}, then the estimator $\eqfk{k}$ defined in  \cref{eq::quadraticestim} satisfies for $n\in\Nz$, $n\geq 2$ and $k\in\pRz$ that
	\begin{align}
		&\sup_{f\in\mathcal{F}(\operatorname{s},L)} \iEx[\denY]{n} [\vert\eqfk{k}-\wqf\vert^2]
		\leq  C (R_n(k)  \vee R^{\operatorname{elbow}}_n ),\label{eq::upperbound-result}
	\end{align}
	for some constant $C>0$ depending on $\denU$ and $ L$.
\end{co}
Let us briefly comment the last result before we state the proof.

\begin{rem}
	Note that the term $R_n(k)$ in the upper bound in \cref{co::rates} depends on parameter $k$. On the other hand, term $R^{\operatorname{elbow}}_n $ does not depend on $k$ and as becomes more clear in \cref{il::rates-non-adaptive} this causes the so-called elbow effect. Here, it becomes also visible that if $\delinf{k}$ and $\delfour{k}$ are bounded in $k$ a parametric rate is retained but otherwise not and that the choice of $k$ needs to depend on $n$.
\end{rem}
\begin{proof}[Proof of \cref{co::rates}]
	From \cref{prop::generalresult} for $f\in\mathcal{F}_{\wF}(\operatorname{s},L)$ and $\operatorname{s}$ increasing follows
	\begin{align*}
		\bias{k}{2} & = \Vert \mathds{1}_{\Rz \setminus [-k,k]} \Mcg{c}{\denX} \Vert_{\Lp[2]{}( \iwF[]{2} )}^4 =	\left\Vert \frac{\operatorname{s}}{\operatorname{s}}  \mathds{1}_{\Rz \setminus [-k,k]}  \Mcg{c}{\denX}  \right\Vert_{\Lp[2]{}(\iwF{2})}^4\\
		& \leq \Vert  \operatorname{s} \Mcg{c}{\denX} \Vert_{\Lp[2]{}}^2  \left\Vert\mathds{1}_{\Rz\setminus[-k,k]}\frac{\iwF{4}}{\operatorname{s}^4 }\right\Vert_{\Lp[\infty]{}} \leq L^4   \left\Vert\mathds{1}_{\Rz\setminus[-k,k]}\frac{\iwF{4}}{\operatorname{s}^4 } \right\Vert_{\Lp[\infty]{}} \leq L^4\frac{\iwF{4}(k)}{\operatorname{s}^4(k)}.
	\end{align*}
	To control $\Lambda_{\denX, \denU}$ first note that if $n\vert \Mcg{c}{\denU}(t)\vert^2\geq \operatorname{s}^{2}(t)$ 
	\begin{align*}
		\frac{1}{n} \frac{\vert \Mcg{c}{\denX}(t)\vert^2}{\vert \Mcg{c}{\denU}(t)\vert^2} \leq  \vert \Mcg{c}{\denX}(t)\vert^2 \operatorname{s}^2(t) \left(\frac{1}{\operatorname{s}^4(t)}\wedge \frac{1}{\operatorname{s}^2(t)n\vert \Mcg{c}{\denU}(t)\vert^2}\right)
		\end{align*}
and if $n\vert \Mcg{c}{\denU}(t)\vert^2< \operatorname{s}^2(t)$
\begin{align*}
	\frac{1}{n} \frac{\vert \Mcg{c}{\denX}(t)\vert^2}{\vert \Mcg{c}{\denU}(t)\vert^2}  \leq \vert \Mcg{c}{\denX}(t)\vert^2\operatorname{s}^2(t) \left(\frac{1}{n^2\vert \Mcg{c}{\denU}(t)\vert^4}\wedge \frac{1}{ s^2(t) n \vert \Mcg{c}{\denU}(t)\vert^2}\right).
\end{align*}
Hence, we obtain for the entire integral
\begin{align*}
	\Lambda_{\denX, \denU}(k) &= \frac{1}{n} \int_{-k}^k\frac{\vert \Mcg{c}{\denX}(t)\vert^2}{\vert \Mcg{c}{\denU}(t)\vert^2} \iwF{4}(t)dt\\
	& \leq \int_{-k}^k \vert \Mcg{c}{\denX}(t)\vert^2 \operatorname{s}^2(t) \left(\frac{\iwF{4}(t)}{\operatorname{s}^4(t)}\wedge \frac{\iwF{4}(t)}{\operatorname{s}^2(t)n\vert \Mcg{c}{\denU}(t)\vert^2}\right)dt \\
	&\qquad +  \int_{-k}^k \vert \Mcg{c}{\denX}(t)\vert^2\operatorname{s}^2(t) \left(\frac{1}{n^2\vert \Mcg{c}{\denU}(t)\vert^4}\wedge \frac{1}{ s^2(t) n \vert \Mcg{c}{\denU}(t)\vert^2}\right)\iwF{4}(t) dt\\
	&\leq L^2 \left\Vert \frac{\iwF{4}}{\operatorname{s}^4}\wedge \frac{\iwF{4}}{\operatorname{s}^2n\vert \Mcg{c}{\denU}\vert^2}\right\Vert_{\Lp[\infty]{}} + \frac{L^2}{n^2}\delinf{k}.
\end{align*}
Note that for $\denX \in \mathcal{F}(\operatorname{s},L)$ we can uniformly bound 
\begin{align*}
	\icstV[\denX|\denU]{2} \leq (\Vert \denU \Vert_{\Lp[\infty]{+}(\basMSy{2c-1})} \Vert \denX \Vert_{\Lp[1]{+}(\basMSy{2c-2})} \vee 1)\leq (L^2\Vert \denU \Vert_{\Lp[\infty]{+}(\basMSy{2c-1})} \vee 1) .
\end{align*}
Consequently, using \cref{prop::generalresult}  we have that
\begin{align*}
	&\sup_{f\in\mathcal{F}(\operatorname{s},L)} \iEx[\denY]{n} [\vert\eqfk{k}-\wqf\vert^2] \\
	&\leq 2  \cstC[\denU]  \icstV[\denX|\denU]{2} \left(  L^4 \frac{\iwF{4}(k)}{\operatorname{s}^4 (k)} + L^2 \left\Vert \frac{\iwF{4}}{\operatorname{s}^4}\wedge \frac{\iwF{4}}{\operatorname{s}^2n\vert \Mcg{c}{\denU}\vert^2}\right\Vert_{\Lp[\infty]{}} +  (L^2 +1) \frac{ \delinf{k}\vee \delfour{k}}{n^2}\right).
\end{align*}
The claim follows with $C= 6\cstC[\denU](L^2\Vert \denU \Vert_{\Lp[\infty]{+}(\basMSy{2c-1})} \vee 1) (L^4 + 2L^2 +1)$.
\end{proof}

\begin{il}\label{il::rates-non-adaptive}
	We illustrate the order of the upper bound in \cref{co::rates} setting
	\begin{align*}
	r^2(\mathcal{F}(\operatorname{s},L)):= \inf_{k\in\pRz}\sup_{f\in\mathcal{F}(\operatorname{s},L)} \iEx[\denY]{n} [\vert\eqfk{k}-\wqf\vert^2]
	\end{align*}	 
	under typical regularity assumptions. For this we define
	\begin{align}\label{eq::k-oracle}
		k_* \in \arg\inf_{k\in\pRz} R_n(k)  =  \arg\inf_{k\in\pRz} \left(  \frac{\iwF{4}(k)}{\operatorname{s}^4(k)} \vee \frac{\delinf{k} \vee  \delfour{k} }{n^2} \right).
	\end{align}
	By inserting $k_*$ into the right-hand side of \cref{eq::upperbound-result} in \cref{co::rates} we derive the order of an upper bound of $r^2(\mathcal{F}(\operatorname{s},L)) $. As in \cite{Miguel2021} and \cite{Brenner2021}, we consider the following examples.
	Concerning the class $\mathcal{F}(\operatorname{s},L)$ defined in  \cref{eq::function-class-general} we distinguish two behaviors of $\operatorname{s}$. Namely the \textit{ordinary smooth case}, i.e. $\operatorname{s}(t) = (1+t^2)^{s/2}$ for $t\in\Rz$ and some $s>0$, and the \textit{super smooth} case, i.e. $\operatorname{s}(t) = \exp(\vert t\vert^r)$ for $t\in\Rz$ some $r>0$. For two real-valued functions $f,g\colon\Rz\rightarrow\Rz $ we write $f(t)\sim g(t)$ if there exist constants $C,\tilde{C}$ such that for all $t\in\Rz$ we have $f(t)\leq C g(t)$ and $g(t)\leq \tilde{C} f(t)$. We then call $g$ the order of $f$. Regarding the error density $\denU$, we distinguish again between two cases. We either assume for some decay parameter $p\in\pRz$  its \textit{(ordinary) smoothness}, i.e.  $\vert \Mcg{c}{\denU}(t) \vert \sim (1+t^2)^{-p/2}$;
	 or its \textit{super smoothness} for some parameter $\sigma\in\pRz$, i.e. $\vert \Mcg{c}{\denU}(t)  \vert \sim \exp( - \vert t\vert^{\sigma})$.
	We restrict our discussion on the examples $\wF(t) \sim t^a$ for $a\in\Rz$. 
	\cref{ass::classes} imposes the condition $s>a$ on the parameters in case of ordinary smoothness of the unknown density $\denX$.
	Note that all the examples satisfy \cref{ass::classes} for $t$ large enough. The order of the upper bound is given in \cref{tab:risks-general} for the cases where both, the unknown density $\denX$ and the error density $\denU$ are ordinarily smooth (first line), or one of them is ordinarily smooth and one is super smooth (line two and three).
	
	\begin{table}[htbp]
		\centering
		\caption{Order of the upper bound for  $\wF(t)\sim t^{a}$ for $t\in\Rz$, $a\in\Rz$}
		\label{tab:risks-general}
		\begin{tabularx}{\textwidth}{c c c c c}
			\toprule
			$\operatorname{s}(t)$&$\vert \Mcg{c}{\denU}(t)  \vert $ & $R_n(k_*)$& $R^{\operatorname{elbow}}_n$	&$r^2(\mathcal{F}(\operatorname{s},L))$ \\
			\midrule
			$(1+t^2)^{\frac{s}{2}}$ & $(1+t^2)^{-\frac{p}{2}}$ & $n^{-\frac{8(s-a)}{4s+4p +1}}$  &  $n^{-\frac{8(s-a)}{4s+4p+1} \wedge 1}$ &$\left\{\begin{array}{l r}n^{-1}, &\text{\footnotesize $s  \geq p+ 2a +\frac{1}4$}\\
			n^{-\frac{8(s-a)}{4s+4p+1}}, &\text{\footnotesize $s < p+ 2a +\frac{1}4 $}\end{array}\right.$\\
			$(1+t^2)^{\frac{s}{2}}$  & $\exp( - \vert t\vert^{\sigma})$&$(\log n)^{-\frac{4(s-a)}{\sigma}}$ & $(\log n)^{-\frac{4(s-a)}{\sigma}}$  &$(\log n)^{-\frac{4(s-a)}{\sigma}}$ \\
			$\exp(\vert t\vert^r)$& $(1+t^2)^{-\frac{p}{2}}$ & $\frac{1}{n^2} (\log n)^{\frac{4(p+a)+1}{r}}$  &$n^{-1}$& $n^{-1}$ \\
			\bottomrule
		\end{tabularx}
	\end{table}
	These examples include all cases considered in \cref{il::example-parameters}. In the ordinary smooth case, $\mathcal{F}(\operatorname{s},L)$ corresponds to the Mellin-Sobolev space, see Definition 2.3.9, \cite{BrennerMiguelDiss}. 
	We omit the case that both $\denX$ and $\denU$ are super smooth, since there are multiple possibilities how the rates behave, depending on the parameters. For a more detailed discussion of this case in the additive convolution model see for example \cite{Lacour2006}.

	\begin{itemize}
		\item[(i)] In the case $a=0$ we have $\wF(t) = 1$ for all $t\in\Rz$. Recall that in this case $\wqf$ equals a (weighted) $\Lp[2]{+}$-norm of the density $\denX$ itself.  
		\item[] 
		Note that the rates correspond to the rates derived by \cite{Butucea2007} for the additive convolution model and \cite{SchluttenhoferJohannes2020b,SchluttenhoferJohannes2020} for the circular convolution model. Both have shown that in the respective cases the rates are minimax. This suggests that this is also the case in the multiplicative convolution model. However, the proof of lower bounds is delayed to future work.

		\item[(ii)] The case  $a=-1$ covers quadratic functional estimation of the survival function $S$ of $X$, more precisely, we have $\iwF{2}(t)= \frac{1}{(c-1)^2+4\pi t^2}$ for all $t\in\Rz$.

		\item[(iii)] The case $a={\beta\in\Nz}$ covers quadratic functionals of derivatives $D^\beta[\denX]$. 

		\end{itemize}

\end{il}

\section{Data-driven estimation}\label{sec::Adaptive}

The optimal choice $k_*\in\pRz$, see  \cref{eq::k-oracle}, for estimator $\widehat{\theta}_{k_*}$, defined in  \cref{eq::quadraticestim}, depends on regularity parameter $\operatorname{s}$ of the unknown density $\denX$, which is not known in general. This motivates the consideration of data-driven procedures. 
The data-driven method is inspired by a bandwidth selection method in kernel density estimation proposed
in  \cite{GL2011}.
Inspired by Lepski's method (which appeared in a
series of papers by \cite{Lepskij1990a,Lepskij1991a,Lepskij1992a,Lepskij1992b}) define, 
given an upper bound $M\in\Nz$ and a sequence of penalties $(V(k))_{k\in\Nz}$, the contrast
\begin{align}\label{eq::contrast}
	A(k) := \max_{k'\in\nset{M}} \left( \vert\eqfk{k\wedge k'} - \eqfk{k'} \vert^2 - V(k') - V(k) \right)_{+},
\end{align}
where we write $a_+:= a\vee 0$ for $a\in\Rz$. In the spirit of \cite{GL2011} combining the contrast
given in  \cref{eq::contrast} and the penalization approach of model
selection  introduced by  \cite{BarronBirgeMassart1999} (for  an extensive
overview of model selection by penalized contrast the reader may
refer to \cite{Massart2007}) we select the dimension
\begin{align}\label{eq::dimension}
	\kest := \arg \min_{k\in\nset{M}} \left\{ A(k) +  V(k) \right\}.
\end{align}
The data-driven estimator of $\wqf$ is given by $\eqfk{\kest}$. We derive an upper bound of its mean squared error $	\iEx[\denY]{n} \left[\vert \eqfk{\kest} - \wqf\vert^2 \right]$. This goal is achieved in two steps. First, we introduce in \cref{sec::partially-data-driven} a penalty series which still depends on the unknown error density $\denX$ and derive an upper bound for the mean squared error of the resulting partially data-driven estimator.
In a second step, we estimate the introduced penalty and therefore propose a fully data-driven estimator in \cref{sec::fully-data-driven} for which the upper bound of its mean squared error is derived based on the result for the partially data-driven estimator.

\subsection{Partially data-driven penalty}\label{sec::partially-data-driven}
In this section, we introduce a partially data-driven penalty. More precisely, for $k\in\Nz$ we set for some numerical constant $\kappa >0 $
\begin{align}
	V(k) &:= \frac{\kappa  \log(\rate{k})}{n^2} \left(\delfour{k} \vee (\log\rate{k})\delinf{k} \right)\nonumber\\
	&\qquad\cdot \left(	 \icstC[\denU]{2} \icstV[\denX|\denU]{2} + \left\vert1 \vee \frac{\rate{k} (\log\rate{k}) ((2k)\vee (\log\rate{k}))}{n} \right\vert^2\right),\label{eq::pen}
\end{align}
where $ \cstV[\denX|\denU] = (\Vert \denY\Vert_{\Lp[1]{+}({\basMSy{2c-2}})} \vee 1)$,
\begin{align}\label{eq::rate-k-omega}
	\rate{k} :=  2k^3 \left( 1\vee n^{-1} \delinf{k} \right),
\end{align}
$\delfour{k}$ was defined in  \cref{eq::Delta-Lambda} and $\delinf{k}$ in  \cref{eq::delta-inf}.
Note that $V(k)$ depends on $\cstV[\denX|\denU]$ and, therefore, on $\denX$. Consequently, it is unknown and needs to be estimated in a second step, see \cref{sec::fully-data-driven}. In addition, $V(k)$ depends on the sample size $n$. However, for sake of simplicity we will omit additional subscripts. We denote by $\mathcal{C}$ a universal numerical constant which might change from line to line.

\begin{thm}\label{thm::adaptiveupper}
	 Under \cref{ass:well-definedness} the estimator $\eqfk{\kest}$ in  \cref{eq::quadraticestim} with $\kest$ in  \cref{eq::dimension} and arbitrary $M\in\Nz$ satisfies that there exists a universal numerical constant $\mathcal{C} $ such that 
	\begin{align*}
	&\iEx[\denY]{n} \left[\vert \eqfk{\kest} - \wqf\vert^2 \right] \\
	&\leq \mathcal{C} \min_{k\in\nset{M}} \left(\bias{k}{2} + V(k) + \frac{\cstC[\denU]  \icstV[\denX|\denU]{2}}{n}\Lambda_{\denX, \denU}(k)\right) + \mathcal{C}\frac{\icstC[\denU]{3} \icstV[\denX|\denU]{3} }{n}(1 \vee  (\Ex[\denY] [Y_1^{4(c-1)}])^2),
\end{align*}
	for  $\Lambda_{\denX, \denU}(k)$ defined  in  \cref{eq::Delta-Lambda},  $V(k)$ in  \cref{eq::pen} with some numerical constant $\kappa>0$.
\end{thm}
Before giving a proof we discuss the result and corresponding rates.
\begin{rem}\label{rem::discussion-adaptive}
	Let us compare the partially data-driven result \cref{thm::adaptiveupper} with \cref{prop::generalresult}. For this, we define
	\begin{align}
		\rho_{k,n} := \log(\rate{k}) \left\vert1 \vee \frac{\rate{k} \log(\rate{k}) ((2k)\vee \log(\rate{k}))}{n} \right\vert^2 \label{eq::rho-price}
	\end{align}
	and
	\begin{align}
		\Vbar(k)&:=  \rho_{k,n} \left(\delfour{k} \vee \log(\rate{k}) \delinf{k} \right).\label{eq::V-bar}
	\end{align}
	Note that $\Vbar(k)$ depends on $\denU$ but not on $\denX$.
	The rate in \cref{thm::adaptiveupper} is determined by 
	\begin{align*}
		R_n^{\operatorname{pd}} &:=  \min_{k\in\pRz} \left(\bias{k}{2}+ \frac{\Vbar(k)}{n^2} + \frac{\Lambda_{\denX, \denU}(k)}{n} \right).
	\end{align*}
	Compared to the upper bound in \cref{prop::generalresult} which is given by 
	\begin{align*}
		 \min_{k\in\pRz} \left(\bias{k}{2} +   \frac{ \delfour{k}}{n^2} + \frac{\Lambda_{\denX, \denU}(k)}{n} \right)
	\end{align*}
	the price we pay for a data-driven approach is on the one hand the additional factor $\rho_{k,n}$ of  \cref{eq::rho-price} and on the other hand the term $\log(\rate{k}) \delinf{k} $. We will see that the last term is often negligible with respect to $\delfour{k}$. For details, refer to the discussion of the rates below in \cref{il::rates-adative}. In addition, we minimize the parameter $k$ over the set $\nset{M}$ instead of $\pRz$. If the minimum is attained in this set, there is no additional deterioration of the rate. The following result provides a data-driven choice for $M$.
\end{rem}
\begin{co}[Upper bound]\label{co::upperbound-M} 
	Let \cref{ass:well-definedness} be satisfied and using the notation of \cref{rem::discussion-adaptive}. For $n\in\Nz$, set 
	\begin{align*}
	M^n_\denU := \max_{k\in\Nz}(\Vbar(k)\leq n^2 \Vbar(1))
	\end{align*}
	and $\kest$ defined in  \cref{eq::dimension} choosing $M =M^n_\denU$. 
	 Then, from \cref{thm::adaptiveupper} we immediately obtain that there exist constants $C_1,C_2>0$ such that for all $n\in\Nz$ with  $\Vbar(1)\geq R_n^{\operatorname{pd}}$ 
	 \begin{align*}
		\iEx[\denY]{n} \left[ \vert \eqfk{\kest} - \wqf \vert^2 \right] \leq C_1 R_n^{\operatorname{pd}} + \frac{C_2}{n}.
	 \end{align*}
\end{co}

\begin{proof}[Proof of \cref{co::upperbound-M}]
	Denote
	\begin{align*}
		k_n &:= \argmin_{k\in\Nz} \left(\bias{k}{2}+ \frac{\Vbar(k)}{n^2} + \frac{1}{n}\Lambda_{\denX, \denU}(k) \right).
	\end{align*}
For $n\in\Nz$ with  $\Vbar(1)\geq R_n^{\operatorname{pd}}$  we have that
\begin{align*}
	\Vbar(1) \geq R_n^{\operatorname{pd}} \geq \frac{\Vbar(k_n)}{n^2}
\end{align*}
from which follows that $k_n\in\nset{M_\denU^n}$ by definition of $M_\denU^n$. With \cref{ass:well-definedness} (ii), we get that for the term $V(k)$ defined in  \cref{eq::pen} there exists a constant $C>0$ such that $ V(k_n) \leq C\Vbar(k_n)n^{-2}$ and with the calculations of  \cref{eq::error-dec-adaptive} follows the result.
\end{proof}
Define
\begin{align}\label{eq::rates-term-2}
	R_n^{\operatorname{fd}}(k) := \frac{\iwF{4}(k)}{\operatorname{s}^4(k) }  \vee \frac{\rho_{n,k}}{n^2}  \left( \log(\rate{k})\delinf{k} \vee \delfour{k}  \right).
\end{align}
Analogously to \cref{co::rates} one can show the following result. We omit the proof.
\begin{co}\label{co::rates-adaptive}
	Under Assumptions \cref{ass:well-definedness} and \cref{ass::classes}, if $f\in\mathcal{F}(\operatorname{s},L)$, see  \cref{eq::function-class-general}, and $L \geq \Ex[\denY] [Y_1^{4(c-1)}]$  then the estimator $\eqfk{\kest}$ in  \cref{eq::quadraticestim} satisfies for $\kest$ in  \cref{eq::dimension}, $k\in\nset{M}$ and $n\in\Nz$, $n\geq 2$ that
	\begin{align*}
		&\sup_{f\in\mathcal{F}(\operatorname{s},L)} \iEx[\denY]{n} [\vert\eqfk{\kest}-\wqf\vert^2] \leq C \left(R_n^{\operatorname{fd}}(k) \vee  R^{\operatorname{elbow}}_n \right)
	\end{align*}
	for some constant $C>0$ depending on $L$ and $\denU$, $ R^{\operatorname{elbow}}_n$ in \cref{eq::rates-elbow-dim} and $R_n^{\operatorname{fd}}(k)$ in  \cref{eq::rates-term-2}.
\end{co}
 Before illustrating the result let us give a few remarks.

\begin{rem}
	First, see that the term $R^{\operatorname{elbow}}_n $ is already known from result \cref{co::rates}. As already stated in \cref{rem::discussion-adaptive}, we see that the price we pay for a data-driven approach is in the first term of the result, more precisely, in the terms $\rho_{n,k}$ and $\log(\rate{k})\delinf{k}$. Although the latter is often negligible with respect to $\delfour{k}$.
\end{rem}	

\begin{il}\label{il::rates-adative}
	Analogously to \cref{il::rates-non-adaptive} we define \begin{align*}
		R^2(\mathcal{F}(\operatorname{s},L)):= \inf_{k\in\pRz}\sup_{f\in\mathcal{F}(\operatorname{s},L)} \iEx[\denY]{n} [\vert\eqfk{k}-\wqf\vert^2]
	\end{align*}	
	and
	\begin{align*}
		k_{\operatorname{opt}}\in \arg\inf_{k\in\pRz} R_n^{\operatorname{fd}}(k) =  \arg\inf_{k\in\pRz}\left(   \frac{\iwF{4}(k)}{\operatorname{s}^4(k) }  \vee \frac{\rho_{n,k}}{n^2}  \left( \log(\rate{k})\delinf{k} \vee \delfour{k}  \right)\right).
	\end{align*}
	By inserting $k_{\operatorname{opt}}$ into the right-hand side of the bound in \cref{co::rates-adaptive} we derive the order of an upper bound of $R^2(\mathcal{F}(\operatorname{s},L)) $.
	\begin{enumerate}
	\item[(i)] In the case when both the unknown density $\denX$ and the error density $\denU$ are ordinarily smooth (first line of \cref{tab:risks-general-adaptive}, see \cref{il::rates-non-adaptive} for the definitions), we obtain different rates depending on the parameters $p$, $s$ and $a$. 
	First, we see that for $k\in\Nz$
	\begin{align*}
		\rate{k} \lesssim  k^3\left(1 \vee \frac1n k^{4(a+p)}  \right)\leq k^{3 + 4(a+p)}
	\end{align*}
	and, consequently, $\log(\rate{k})\lesssim (1 \vee \log(k))$.
	In this case the term $\delfour{k}$ is dominating $\log(\rate{k})\delinf{k}$ with rate $k^{4(a+p)+1}$ and the bias-term is of order $ k^{-4(s-a)}$, assuming that $a\leq s$ and $4(p+a) > -1$. Consequently, to obtain an upper bound for the optimal choice of $k$ we have to solve 
	\begin{align*}
		1 &\sim \frac{\Vbar{k}}{n^2}  \frac{\operatorname{s}^4(k)}{\iwF{4}(k)} +\sim \left(1 \vee \frac {k^4(1 \vee \frac1n k^{4(a+p)})\log(k)}{n}\right)^2 \frac{\log(k)\vee 1}{n^2} k^{4(s+p) +1}\\
		&=  \left(1 \vee \frac {k^4\log(k)}{n} \vee \frac{k^{4 +4(a+p)}}{n^2}\right)^2 \frac{\log(k)\vee 1}{n^2} k^{4(s+p) +1}.
	\end{align*}
	For simplicity, assume $(p+1)\geq 1$ and $(s-a)\geq \frac{3}{4}$. In this case, we observe again the elbow effect. More precisely,  choosing $k_{\operatorname{opt}}= \left(\frac{n^2}{\log(n)}\right)^{\frac{1}{4s+4p+1}}$ in the case of $(s-a) < (p+a) +\frac{1}4$ this term dominates the term $R^{\operatorname{elbow}}_n$. On the other hand, in the case of $(s-a) < (p+a) +\frac{1}4$  we can retain the rate $n^{-1}$. Slower rates are achieved under other assumptions on the parameters $s,p$ and $a$.
	\item[(ii)] Considering the case that the unknown density $\denX$ is ordinarily smooth and the error density $\denU$ is super smooth (second line of Table ref{tab:risks-general-adaptive}), we see that the rate can be retained and the dimension parameter only changes slightly.
	First, $\delfour{k}$ and $\log(\rate{k}) \delinf{k}$ are both of order $k^{4a}\exp(4k^\sigma)$. The bias is of order $k^{-4(s-a)}$.  Consequently, we need to find $k$ for which holds
	\begin{align*}
		1 \sim \frac{\log(\rate{k})}{n^2}k^{4s} \exp{(4k^\sigma)} \left\vert1 \vee \frac{\rate{k} \log(\rate{k}) ((2k)\vee \log(\rate{k}))}{n} \right\vert^2.
	\end{align*}
	For $k_{\operatorname{opt}} \sim \left(\frac{\log n }{4}- \frac{4s}{\sigma} \log\left(\frac{\log n }{4}\right)\right)^{1/ \sigma}$ we see that
	\begin{align*}
	 \rate{k_{\operatorname{opt}}} &\sim k_{\operatorname{opt}}^3( 1\vee n^{-1} k_{\operatorname{opt}}^{4a}\exp(4k_{\operatorname{opt}}^\sigma))\\
		&\sim \left(\frac{\log(n)}{4}\right)^{\frac{3}{\sigma}} \left( 1 \vee \frac1n \left(\frac{\log(n)}{4}\right)^{\frac{4a}{\sigma}}  \left(\frac{\log(n)}{4}\right)^{\frac{-4s}{\sigma}}  \right)\sim \left(\frac{\log(n)}{4}\right)^{\frac{3}{\sigma}} 
	\end{align*}
	and hence $\log(\rate{k_{\operatorname{opt}}})$ is of order $\log(\log(n))$ fulfilling the above stated condition.
	\item[(iii)] For the unknown density $\denX$ being super smooth and the error density $\denU$ ordinarily smooth, we solve
	\begin{align*}
		1 \sim \frac{\log(\rate{k})}{n^2} k^{4(a+p)+1} k^{-4a} \exp{(4\alpha k^r)} \left\vert1 \vee \frac{\rate{k} \log(\rate{k}) ((2k)\vee \log(\rate{k}))}{n} \right\vert^2.
	\end{align*}
	For $k_{\operatorname{opt}} \sim \left(\frac{\log n}{4}- \frac{4p+1}{4 r} \log\left(\frac{\log n}{4}\right)\right)^{1/ r}$, $\rate{k_{\operatorname{opt}}}$ and $\log(\rate{k_{\operatorname{opt}}})$ behave as in part (ii) with rate $\log(n)$ and $\log(\log(n))$. Similarly to part (ii) the first term is of order 1. In case of $4(a+p)+1 <0$ the rate $n^{-1}$ is retained. Otherwise, there is a log-loss.
	
	 \end{enumerate}
	\begin{table}[htbp]
		\centering
		\caption{Upper bound of the order of the estimation risk for  $\wF(t)\sim t^{a}$ for $t\in\Rz$, $a\in\Rz$}
		\label{tab:risks-general-adaptive}
		\begin{tabularx}{\textwidth}{c c c c c}
			\toprule
			$\operatorname{s}(t)$&$\vert \Mcg{c}{\denU}(t)  \vert $ & $R_n^{\operatorname{fd}}(k_{\operatorname{opt}})$&$R^{\operatorname{elbow}}_n$ &	$R^2(\mathcal{F}(\operatorname{s},L))$ \\
			\midrule
			\footnotesize$(1+t^2)^{\frac{s}{2}}$ & \footnotesize$(1+t^2)^{-\frac{p}{2}}$ &\footnotesize$\left(\frac{n^2}{\log n }\right)^{-\frac{4(s-a)}{4s+4p+1}}$ &  \footnotesize $n^{-\frac{8(s-a)}{4s+4p+1} \wedge 1}$& $\left\{\begin{array}{l r} \text{\footnotesize$n^{-1}$}, &\text{\footnotesize $s\geq p+2a +\frac{1}4$}\\
			\text{\footnotesize$\left(\frac{n^2}{\log n }\right)^{-\frac{4(s-a)}{4s+4p+1}}, $}&\text{\footnotesize $s < p + 2a + \frac{1}4$}\end{array}\right.$\\
			\footnotesize$(1+t^2)^{\frac{s}{2}}$  & \footnotesize$\exp( - \vert t\vert^{\sigma})$ & \footnotesize$(\log n )^{-\frac{4(s-a)}{\sigma}}$&  \footnotesize$(\log n)^{-\frac{4(s-a)}{\sigma}}$&\footnotesize$(\log n )^{-\frac{4(s-a)}{\sigma}}$ \\
			\footnotesize$\exp(\vert t\vert^r)$& \footnotesize$(1+t^2)^{-\frac{p}{2}}$ & \footnotesize$\frac{(\log n)^{\frac{4(a-p)+1}{r}}}{n}$ &\footnotesize$n^{-1}$& 
			\footnotesize$\frac{(\log n)^{\frac{4(a-p)+1}{r}}}{n}$\\
			\bottomrule
		\end{tabularx}
	\end{table}
	
\end{il}

\begin{proof}[Proof of \cref{thm::adaptiveupper}]
	 Using that for $k,k'\in\nset{M}$  it holds 
	 \begin{align*}
		 \vert\eqfk{ k} - \eqfk{k'} \vert^2 - V(k') - V(k)\leq  {A(k\wedge k')} \leq A(k) + A(k'),
	 \end{align*}
	we have for any $k\in\nset{M}$ that
	\begin{align}
	\vert \eqfk{\kest} - \wqf\vert^2 &\leq 2\vert\eqfk{ \kest}- \eqfk{ k}\vert^2 + 2\vert \eqfk{k} - \wqf\vert^2\nonumber\\
	&\leq 2\left(\vert\eqfk{ \kest}- \eqfk{ k}\vert^2 -V(k)+V(k)-V(\kest) + V(\kest)\right)+ 2\vert \eqfk{k} - \wqf\vert^2\nonumber \\
	&\leq 2(A(k) + A(\kest) + V(k) + V(\kest)) + 2\vert \eqfk{k} - \wqf\vert^2\nonumber\\
	&\leq 4 A(k) +4 V(k) + 2\vert \eqfk{k} - \wqf\vert^2. \label{eq::error-dec-adaptive}
\end{align}
	Next, study $\iEx[\denY]{n}[A(k)]$ for any $k\in\nset{M}$. For this, we first decompose $A(k)$ reasonably. Using the decomposition  \cref{eq:decomp-1} we get that
	\begin{align*}
		A(k) &= \max_{k <k'\leq M} \left( \vert\eqfk{k} - \eqfk{k'} \vert^2 - V(k') -V(k) \right)_{+}\\
		&= \max_{k <k'\leq M} \left( \vert\eqfk{k} - \wqfk{k} - (\eqfk{k'}-\wqfk{k'}) + \wqfk{k}- \wqfk{k'}\vert^2 - V(k') -V(k) \right)_{+}\\
		&= \max_{k <k'\leq M} \left( \vert U_{k} -U_{k'} + 2(W_k - W_{k'}) + \wqfk{k}- \wqfk{k'}\vert^2 - V(k') -V(k) \right)_{+}
	\end{align*}
	Further, note that $(\wqfk{k}- \wqfk{k'})^2 \leq \bias{k\wedge k'}{2}$. We subtract and add additionally $5\bias{k}{2}$ which will be later useful to handle the linear part. With this we get 
	\begin{align}
		&A(k) \nonumber\\
		&\leq \max_{k <k'\leq M} \left( 4\vert U_{k}\vert^2 + 4 \vert U_{k'}\vert^2 + 16 \vert W_k - W_{k'}\vert^2 + 4\bias{k}{2} - V(k') -V(k) \right)_{+}\nonumber\\
		&\leq \max_{k <k'\leq M} \left( 4\vert U_{k}\vert^2 + 4 \vert U_{k'}\vert^2 + 16 \vert W_k - W_{k'}\vert^2 \right.\nonumber\\
		&\left. \qquad\qquad\quad - (V(k') +V(k)+ (5-4) \bias{k}{2} ) \right)_{+} + 5 \bias{k}{2}.\label{eq::Ak-bound}
	\end{align}
	Putting the calculations together and using \cref{prop::generalresult} we get that
	\begin{align*}
			\iEx[\denY]{n} &\left[\vert \eqfk{\kest} - \wqf\vert^2 \right]\\
			& \leq 4\iEx[\denY]{n} \left[ \max_{k <k'\leq M}\left( 4\vert U_{k'}\vert^2 -  \frac12 V(k')\right)_+ \right]+ 4\iEx[\denY]{n} \left[ \left( 4\vert U_{k}\vert^2 -   V(k)\right)_+ \right]\\
			&\qquad + 4\iEx[\denY]{n} \left[ \max_{k <k'\leq M}\left( 16\vert W_{k'}- W_k\vert^2 -(\frac12 V(k') + \bias{k}{2})\right)_+ \right]\\
			&\qquad +20 \bias{k}{2} + 2\iEx[\denY]{n} \left[\vert \eqfk{k} - \wqf \vert^2 \right] + 4 V(k)\\
			&\leq 32 \iEx[\denY]{n} \left[ \max_{k \leq k'\leq M}\left( \vert U_{k'}\vert^2 -  \frac18 V(k')\right)_+ \right]\\
			&\qquad + 64\iEx[\denY]{n} \left[ \max_{k <k'\leq M}\left( \vert W_{k'}- W_k\vert^2 -(\frac1{32} V(k') + \frac{1}{16} \bias{k}{2})\right)_+ \right]\\
			&\qquad + 24  \bias{k}{2} + 4  \cstC[\denU]  \icstV[\denX|\denU]{2} \left(  \frac{ \delfour{k}}{n^2} + \frac{\Lambda_{\denX, \denU}(k)}{n}\right) + 4 V(k).
	\end{align*}
	Consequently, we split the contrast term $A(k)$, into a U-statistic part and a linear part. In the second step,  we need to split each term again in a bounded and an unbounded term, see \cref{lem::concentration-ustat} and \cref{lem::concentration-linear}.
	 We do this, to then apply an exponential inequality, in \cref{lem::concentration-bounded-ustat}, and a Bernstein inequality, in \cref{lem::concentration-bounded-lin}. Lastly, we show that both unbounded parts are negligible, in \cref{lem::concentration-unbounded-ustat} and \cref{lem::concentration-unbounded-lin}. Choosing $\kappa\geq 16^2\kappa^*$ with $\kappa^*$ as in concentration inequality \cref{lem::Ustatconcentration} the assumptions of \cref{lem::concentration-ustat} and \cref{lem::concentration-linear} are fulfilled. Thus, applying \cref{lem::concentration-ustat} and \cref{lem::concentration-linear}, and using that $V(k)\geq  \cstC[\denU]  \icstV[\denX|\denU]{2}\delfour{k}n^{-2}$ yields the result. 	 	
 \end{proof}

\subsection{Fully data-driven penalty}\label{sec::fully-data-driven}
Recall that in the definition of $V(k)$, see  \cref{eq::pen}, appears the constant $\icstV[\denX|\denU]{2}$ defined in  \cref{eq::constants} and consequently also in the definition of $\kest$, see  \cref{eq::dimension}. The constant depends on the unknown density $\denX$. For this, first consider the estimator $\widehat{\cstV[]}^2:= 1 + \frac1n \sum_{j\in\nset{n}} Y_j^{4(c-1)}$ which is unbiased for the parameter $\icstV[]{2} := 1 + \Ex[\denY][Y_1^{4(c-1)}]\geq \icstV[\denX|\denU]{2}$. Based on this estimator let us introduce a fully data-driven sequence of penalties $\widehat{V}(k)$ for $k\in\Nz$ given by
\begin{align*}
	\widehat{V}(k) &:= \frac{\kappa  \log(\rate{k})}{n^2} \left(\delfour{k} \vee \log(\rate{k}) \delinf{k} \right)\nonumber\\
	&\qquad\cdot \left(	 \text{{$2\widehat{\cstV[]}^2$}} \icstC[\denU]{2} + \left\vert1 \vee \frac{\rate{k} \log(\rate{k}) ((2k)\vee \log(\rate{k}))}{n} \right\vert^2\right)
\end{align*}
which is now fully known in advance, and fully known upper bound $M_\denU^n\in\Nz$ defined in \cref{co::upperbound-M}. Considering the fully data-driven estimator $\eqfk{\kestf}$ defined in  \cref{eq::quadraticestim} with dimension parameter selected by Goldenshluger and Lepski's method
\begin{align}
	\widehat{A}(k) &:=  \max_{k'\in\nset{M_\denU^n}} \left( \vert\eqfk{k\wedge k'} - \eqfk{k'} \vert^2 - \widehat{V}(k') - \widehat{V}(k) \right)_{+},\nonumber\\
	\kestf &:= \argmin_{k\in\nset{M_\denU^n}} \{  \widehat{A}(k) + \widehat{V}(k)\}.\label{eq::dim-k-fully}
\end{align}
We derive an upper bound for its mean squared error.

\begin{thm}\label{thm::adaptiveupper-data-driven}
	Let $\Ex[\denY][Y_1^{8(c-1)}]< \infty$. Under \cref{ass:well-definedness} the estimator $\eqfk{\kestf}$ defined in  \cref{eq::quadraticestim} with dimension $\kestf$ defined in  \cref{eq::dim-k-fully} satisfies that there exist constant $C_{\denX,\denU}$ and  a universal numerical constant $\mathcal{C}$ such that 
	\begin{align*}
	&\iEx[\denY]{n} \left[\vert \eqfk{\kestf} - \wqf\vert^2 \right] \leq \mathcal{C} \min_{k\in\nset{M_\denU^n}} \left(\bias{k}{2}+ V(k) + \frac{\cstC[\denU]  \icstV[\denX|\denU]{2}}{n}\Lambda_{\denX, \denU}(k)\right) + C_{\denX,\denU}  \frac{1}{n},
\end{align*}
	for $M_\denU^n$ defined in \cref{co::upperbound-M}, $V(k)$ in  \cref{eq::pen} with some numerical constant $\kappa>0$ large enough, and $\Lambda_{\denX, \denU}(k)$ in  \cref{eq::Delta-Lambda}.
\end{thm}
\begin{proof}[Proof of \cref{thm::adaptiveupper-data-driven}]
	The proof follows along the lines of the proof of \cref{thm::adaptiveupper}.
	 First, we note that for all $k\in\nset{M_\denU^n}$ we have
\begin{align*}
	\widehat{A}(k) \leq A(k) + 2 \max_{k\leq k'\leq M_\denU^n} \{ (V(k')- \widehat{V}(k'))_+ \}
\end{align*}
and, thus, similar to  \cref{eq::error-dec-adaptive} follows
\begin{align*}
	\vert \eqfk{\kestf} - \wqf\vert^2 \leq 4 A(k) +4 \widehat{V}(k) + 2\vert \eqfk{k} - \wqf\vert^2 + 8 \max_{k\leq k'\leq M_\denU^n} \{ (V(k')- \widehat{V}(k'))_+ \}
\end{align*}
	Using  \cref{eq::Ak-bound} we get
	\begin{align*}
		\iEx[\denY]{n} &\left[\vert \eqfk{\kestf} - \wqf\vert^2 \right]\\
		&\leq 32 \iEx[\denY]{n} \left[ \max_{k \leq k'\leq  M_\denU^n}\left( \vert U_{k'}\vert^2 -  \frac18 V(k')\right)_+ \right]\\
		&\qquad + 64\iEx[\denY]{n} \left[ \max_{k <k'\leq  M_\denU^n}\left( \vert W_{k'}- W_k\vert^2 -(\frac1{32} V(k') + \frac{1}{16} \bias{k}{2})\right)_+ \right]\\
		&\qquad + 24  \bias{k}{2} + 4  \cstC[\denU]  \icstV[\denX|\denU]{2} \left(  \frac{ \delfour{k}}{n^2} + \frac{\Lambda_{\denX, \denU}(k)}{n}\right) \\
		&\qquad + 4 \iEx[\denY]{n}[\widehat{V}(k)]+ 8 \iEx[\denY]{n} [ \max_{k\leq k'\leq M_\denU^n} \{ (V(k')- \widehat{V}(k'))_+ \} ].
\end{align*}
	 The first two summands we again control with \cref{lem::concentration-ustat} and \cref{lem::concentration-linear}.
	 Choosing $\kappa\geq 16^2\kappa^*$ with $\kappa^*$ as in concentration inequality \cref{lem::Ustatconcentration} the assumptions of \cref{lem::concentration-ustat} and \cref{lem::concentration-linear} are fulfilled. Thus, applying \cref{lem::concentration-ustat} and \cref{lem::concentration-linear} we get
	 \begin{align*}
		\iEx[\denY]{n} &\left[\vert \eqfk{\kestf} - \wqf\vert^2 \right]\\
		&\leq  \mathcal{C}\bias{k}{2} + 4  \cstC[\denU]  \icstV[\denX|\denU]{2} \left(  \frac{ \delfour{k}}{n^2} + \frac{\Lambda_{\denX, \denU}(k)}{n}\right) +\mathcal{C}\frac{\icstC[\denU]{3} \icstV[\denX|\denU]{3} }{n}(1 \vee  (\Ex[\denY] [Y_1^{4(c-1)}])^2)\\
		&\qquad + 4 \iEx[\denY]{n}[\widehat{V}(k)]+ 8 \iEx[\denY]{n} [ \max_{k\leq k'\leq M_\denU^n} \{ (V(k')- \widehat{V}(k'))_+ \} ].
\end{align*}
Moreover, we have $\iEx[\denY]{n}[\widehat{V}(k)]\leq 4 V(k)$   and using that $V(k)\geq  \cstC[\denU]  \icstV[\denX|\denU]{2}\delfour{k}n^{-2}$ yields 
\begin{align}
	\iEx[\denY]{n} &\left[\vert \eqfk{\kestf} - \wqf\vert^2 \right]\nonumber\\
	&\leq  \mathcal{C}\bias{k}{2} + \mathcal{C}   \left(  V(k) +\cstC[\denU]  \icstV[\denX|\denU]{2} \frac{\Lambda_{\denX, \denU}(k)}{n}\right) +\mathcal{C}\frac{\icstC[\denU]{3} \icstV[\denX|\denU]{3} }{n}(1 \vee  (\Ex[\denY] [Y_1^{4(c-1)}])^2)\nonumber\\
	&\qquad + 8 \iEx[\denY]{n} [ \max_{k\leq k'\leq M_\denU^n} \{ (V(k')- \widehat{V}(k'))_+ \} ].\label{eq::bound-3}
\end{align}
	Making use of $\Vbar(k)$ defined in  \cref{eq::V-bar} and $\max_{k\leq k'\leq M_\denU^n}\Vbar(k')\leq n^2\Vbar(1)$ by the definition of $M_{\denU}^n$ we obtain 
\begin{align*}
	&\max_{k\leq k'\leq M_\denU^n} \{ (V(k')- \widehat{V}(k'))_+ \} \\
	&\leq \icstC[\denU]{2} \kappa \left( \Ex[\denY][Y_1^{4(c-1)}] - \frac{2}{n} \sum_{j\in\nset{n}} Y_j^{4(c-1)} \right)_+ \frac{1}{n^2} \max_{k\leq k'\leq M_\denU^n}\Vbar(k') \\
	&\leq\icstC[\denU]{2} \kappa \Vbar(1) \left( \Ex[\denY][Y_1^{4(c-1)}] - \frac{2}{n} \sum_{j\in\nset{n}} Y_j^{4(c-1)} \right)_+.
\end{align*}
For $a>0$ and $b\geq 0$ it holds that $(\frac{a}{2}-b)_+ \leq 2 \frac{\vert a-b\vert^2}{a}$.
Therefore, it follows
\begin{align*}
	\iEx[\denY]{n} \left[ \max_{k\leq k'\leq M_\denU^n} \{ (V(k')- \widehat{V}(k'))_+ \} \right] &\leq \icstC[\denU]{2} \kappa \Vbar(1)\iEx[\denY]{n}[\left( \Ex[\denY][Y_1^{4(c-1)}] - \frac{2}{n} \sum_{j\in\nset{n}} Y_j^{4(c-1)} \right)_+ ]\\
	&\leq \frac{\mathcal{C}}{n} \frac{\icstC[\denU]{2} \Vbar(1) \Ex[\denY][Y_1^{8(c-1)}]}{\Ex[\denY][Y_1^{4(c-1)}]}.
\end{align*}
The last bound together with  \cref{eq::bound-3} and  setting $C_{\denX,\denU} = \mathcal{C}\icstC[\denU]{3} \icstV[\denX|\denU]{3}(1 \vee  (\Ex[\denY] [Y_1^{4(c-1)}])^2+ \frac{	 \Vbar(1) \Ex[\denY][Y_1^{8(c-1)}]}{\Ex[\denY][Y_1^{4(c-1)}]})$ concludes the proof.
\end{proof}

\begin{rem}
	It is remarkable that the upper bounds in  \cref{thm::adaptiveupper} and \cref{thm::adaptiveupper-data-driven} only differ in the constants, i.e $C_{\denX,\denU}$ compared to $\mathcal{C}\icstC[\denU]{3} \icstV[\denX|\denU]{3} (1 \vee  (\Ex[\denY] [Y_1^{4(c-1)}])^2)$.
	However, in \cref{thm::adaptiveupper-data-driven} we impose only a higher moment assumption, i.e $\Ex[\denY][Y_1^{8(c-1)}]< \infty$ compared to $\Ex[\denY][Y_1^{4(c-1)}]< \infty$. Consequently, for the fully-data driven case we obtain under the assumptions of \cref{co::rates-adaptive} that
	\begin{align*}
		&\sup_{f\in\mathcal{F}(\operatorname{s},L)} \iEx[\denY]{n} [\vert\eqfk{\kestf}-\wqf\vert^2] \leq C \left(R_n^{\operatorname{fd}}(k) \vee  R^{\operatorname{elbow}}_n \right)
	\end{align*}
	and the same rates as for the partially data-driven case, see \cref{il::rates-adative}.
\end{rem}

\section{Numerical Study}\label{sec::Numerical-Study}
In this section, we illustrate the behavior of the estimator $\eqfk{k}$ presented in  \cref{eq::quadraticestim} and the fully data-driven choice $\widehat{k}$ given in  \cref{eq::dim-k-fully}. To do so we consider the following example continuing with \cref{il::example-parameters} (i). That is,we set $c=0.5$ and $\wF(t)=1$ for $t\in\Rz$. Recall that in this case we have $\wqf = \Vert f\Vert_{\Lp[2]{+}}^2$. For illustration purposes, we consider one example of an ordinarily smooth density, where $X$ is Beta$(a,b)$-distributed with parameters $a  = 2, b = 1$, and one example of a super smooth density, where $X$ is log-normally distributed with parameters $\mu  =0, \sigma^2=1$ . More precisely,
\begin{align*}
    \denX_1(x)&=2x\mathds{1}_{(0,1)}(x),\\
    \denX_2(x)&= \frac{1}{\sqrt{2\pi x^2}} \exp(-\log(x)^2/2).
\end{align*}
For error densities, we consider the following two examples of $U$ being either Pareto-distributed or log-normally distributed, i.e.
\begin{align*}
    \denU_1(x)&= \mathds{1}_{(1, \infty)}(x)x^{-2},\\
    \denU_2(x)&= \frac{1}{\sqrt{2\pi x^2}} \exp(-\log(x)^2/2).
\end{align*}
The corresponding Mellin-transforms are given by
\begin{align*}
    \Mcg{\frac12}{\denX_1}(t)&= 2/(1.5 + 2\pi it),\\
    \Mcg{\frac12}{\denX_2}(t)&= \exp\left((-0.5+2\pi i t)^2 /2\right),\\
    \Mcg{\frac12}{\denU_1}(t)&=(1.5-2\pi it)^{-1},\\
    \Mcg{\frac12}{\denU_2}(t)&= \exp\left((-0.5+2\pi i t)^2 /2\right).
\end{align*}
For a detailed discussion on the Mellin transforms and their decaying behavior of common probability densities, see \cite{BrennerMiguelDiss}. We illustrate our results by considering the following three cases.
\begin{figure}
    \centering
    \begin{subfigure}[b]{0.35\textwidth}
        \centering
        \includegraphics[width=\textwidth]{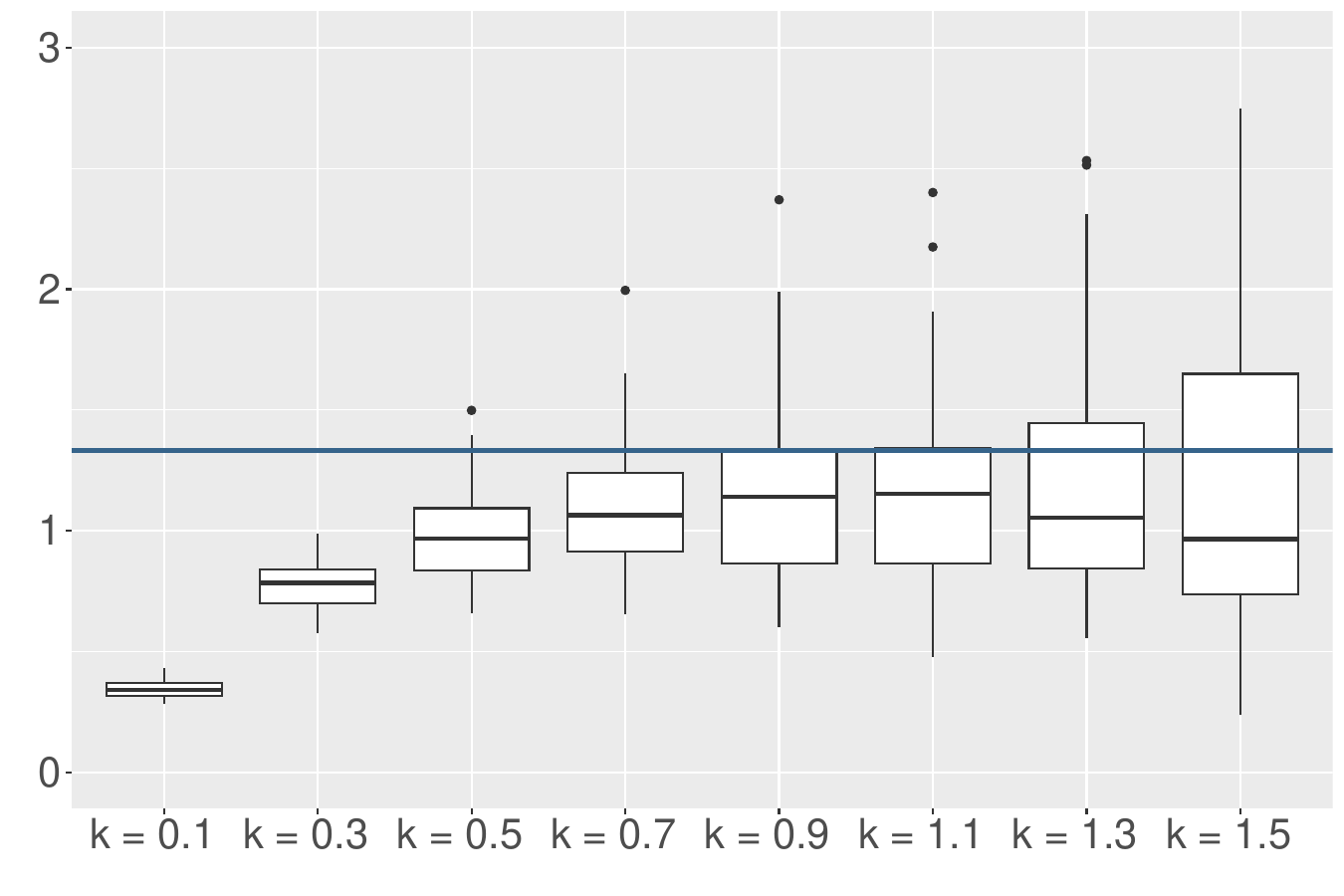}
        \caption{Example (o.s. - o.s.), $n=100$}
        \label{fig:sub1}
    \end{subfigure}
    \begin{subfigure}[b]{0.35\textwidth}
        \centering
        \includegraphics[width=\textwidth]{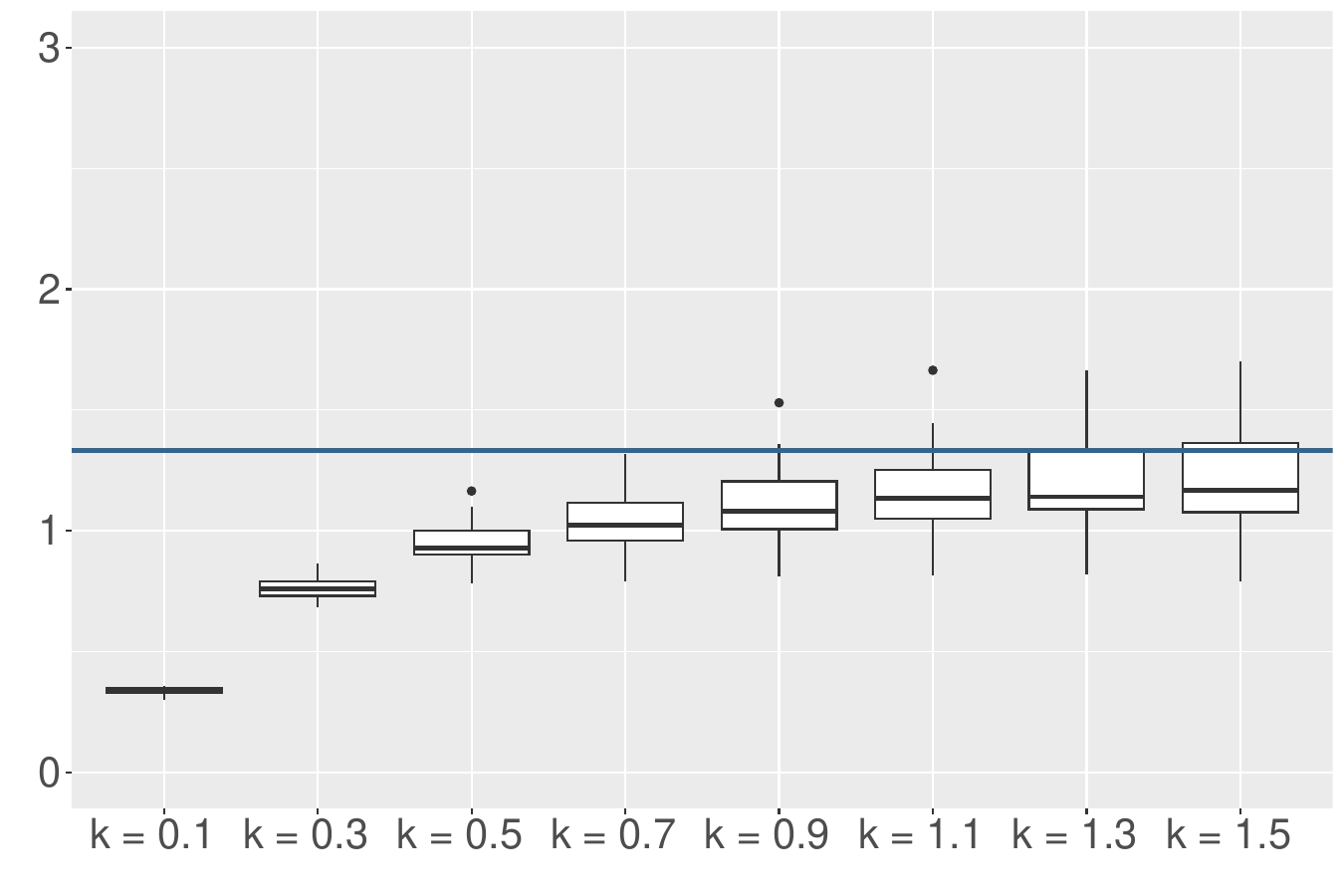}
        \caption{Example  (o.s. - o.s.), $n=500$}
        \label{fig:sub2}
    \end{subfigure}\\
    \begin{subfigure}[b]{0.35\textwidth}
        \centering
        \includegraphics[width=\textwidth]{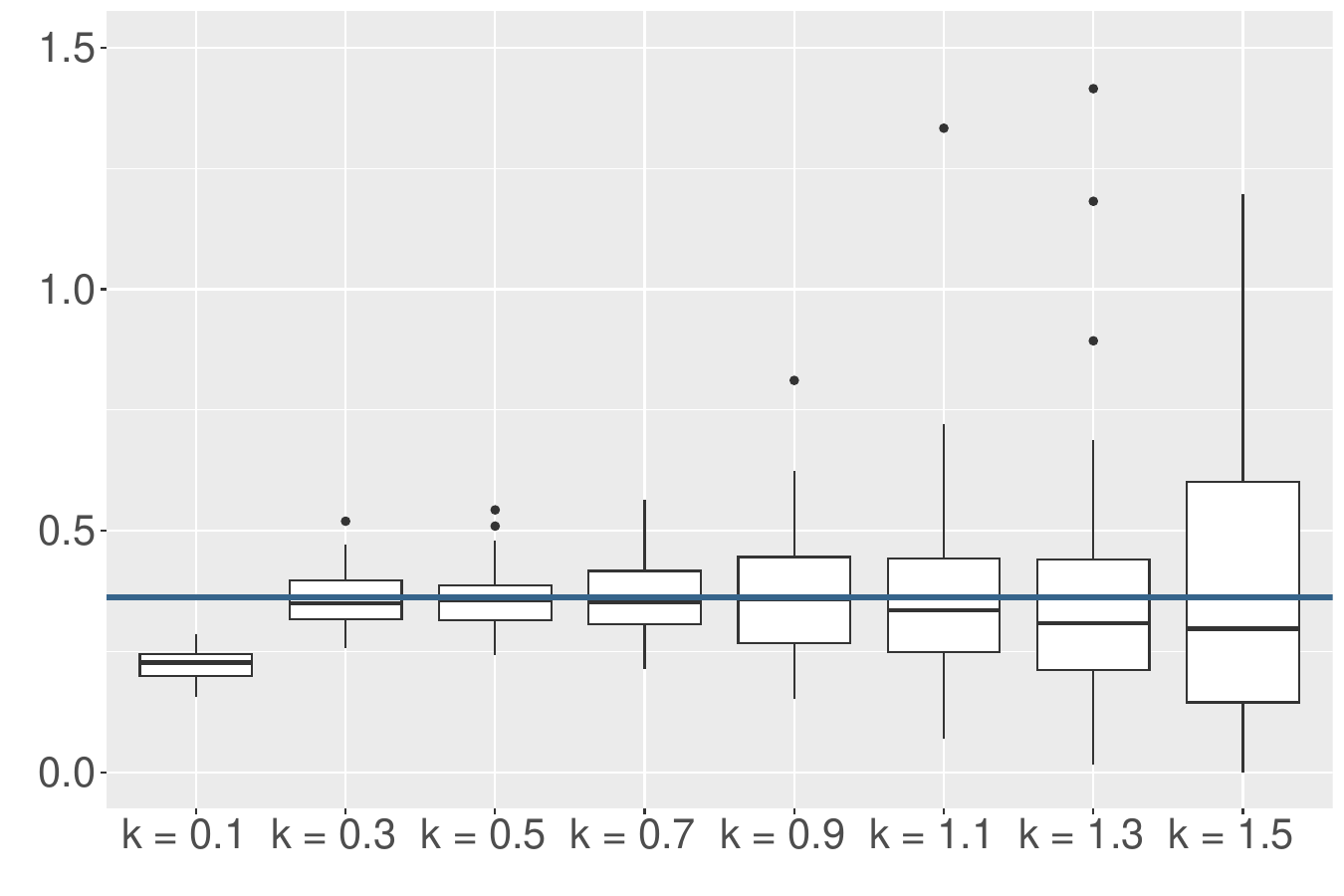}
        \caption{Example (s.s. - o.s.), $n=100$}
        \label{fig:sub3}
    \end{subfigure}
    \begin{subfigure}[b]{0.35\textwidth}
        \centering
        \includegraphics[width=\textwidth]{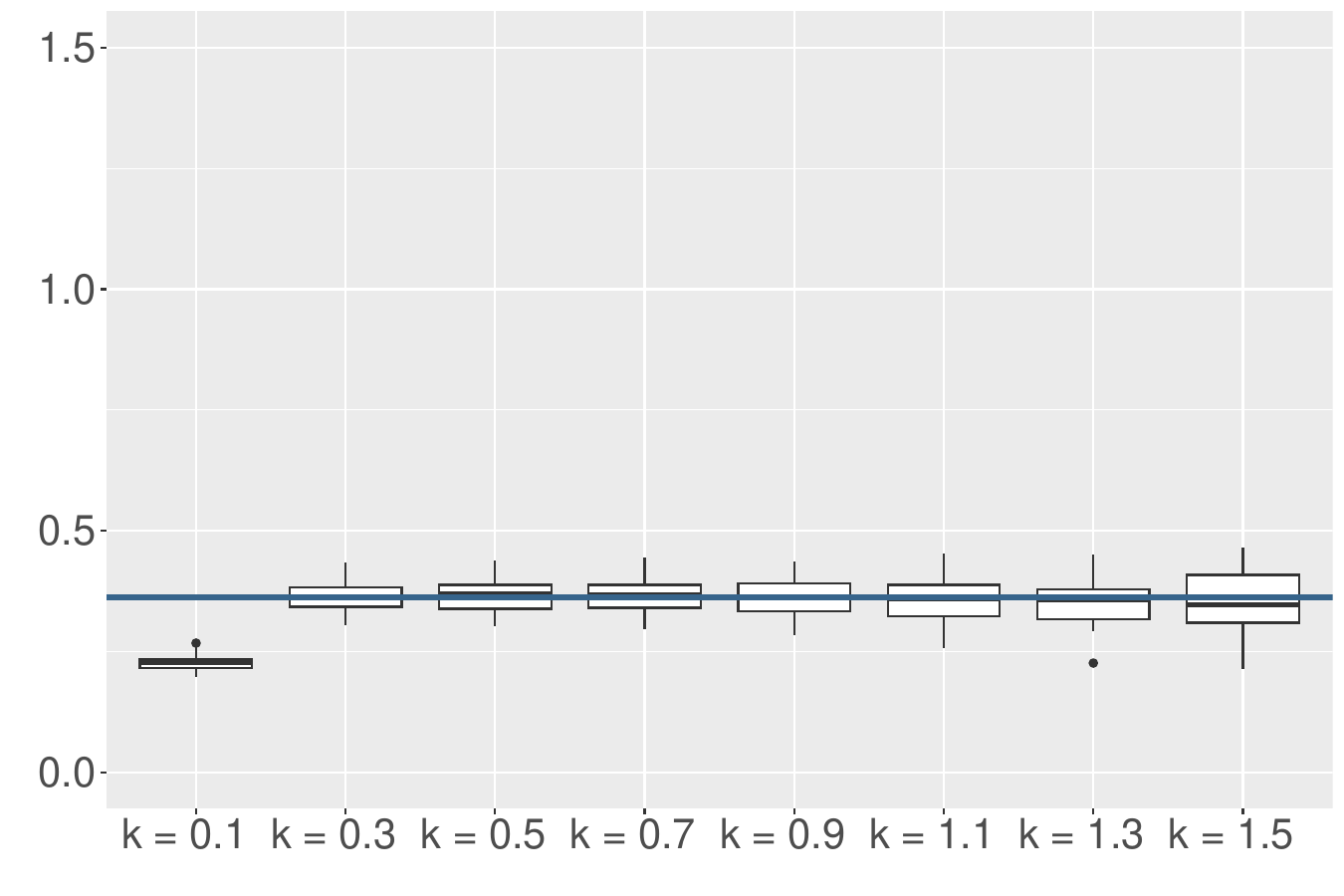}
        \caption{Example  (s.s. - o.s.), $n=500$}
        \label{fig:sub4}
    \end{subfigure}\\
    \begin{subfigure}[b]{0.35\textwidth}
        \centering
        \includegraphics[width=\textwidth]{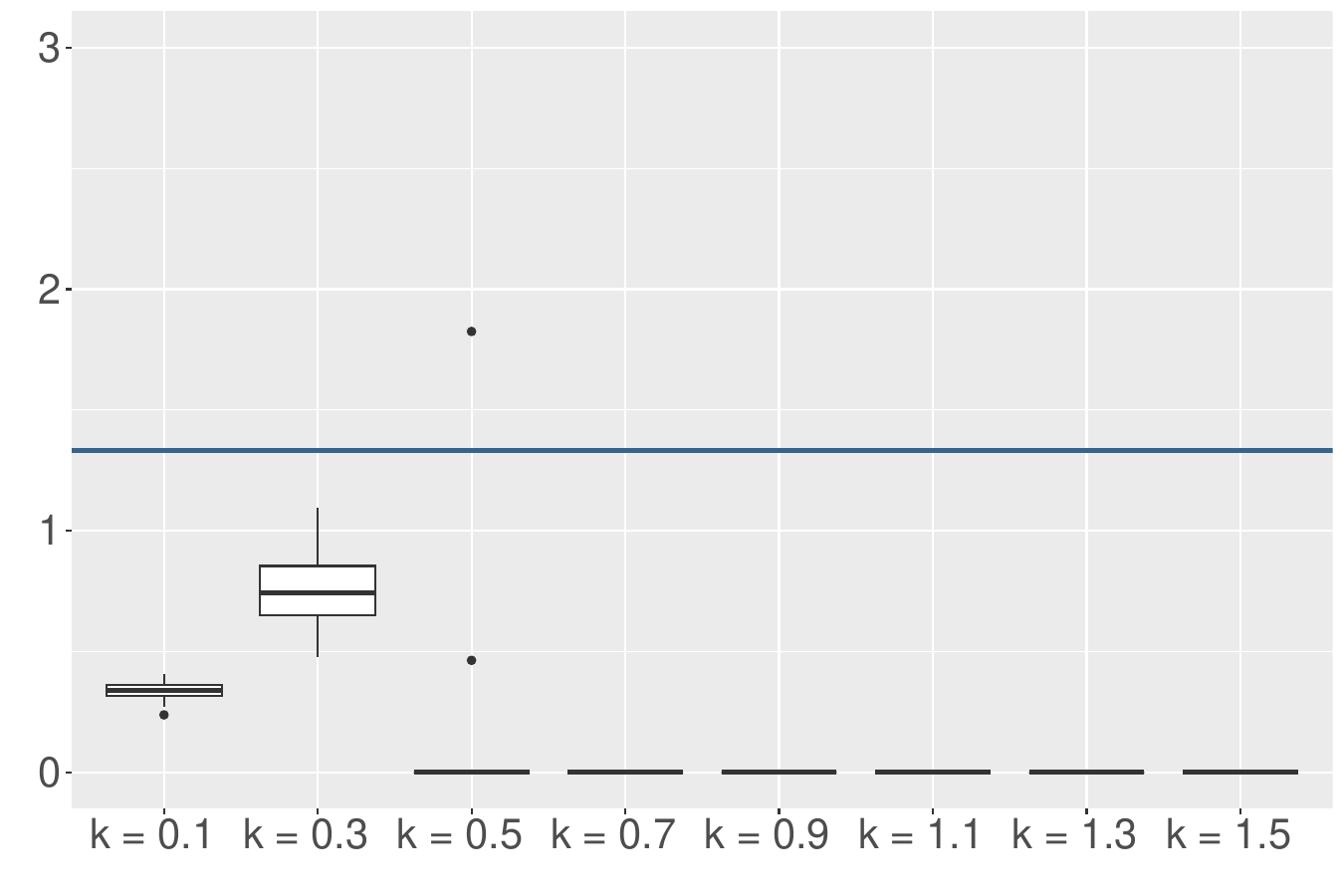}
        \caption{Example (o.s. - s.s.), $n=100$}
        \label{fig:sub5}
    \end{subfigure}
    \begin{subfigure}[b]{0.35\textwidth}
        \centering
        \includegraphics[width=\textwidth]{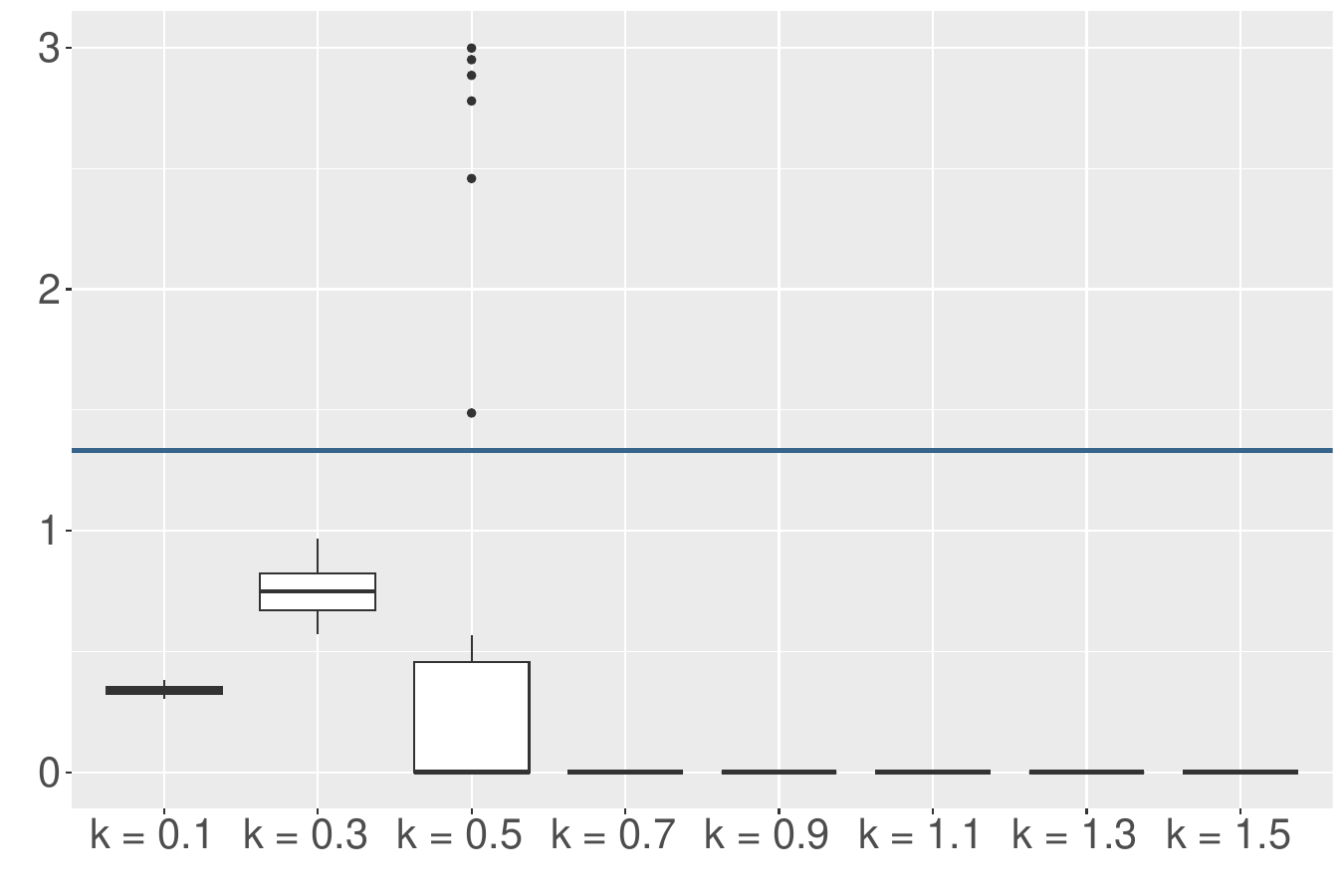}
        \caption{Example  (o.s. - s.s.), $n=500$}
        \label{fig:sub6}
    \end{subfigure}
    \caption{The boxplots represent the values of $\eqfk{k}$ for Examples  (o.s. - o.s.), (s.s. - o.s.) and (o.s - s.s.) over 50 iterations. The horizontal lines indicate $\wqf$.}
    \label{fig::estimation-different-dimensions}
\end{figure}

\begin{figure}
    \centering
    \includegraphics[width=0.6\textwidth]{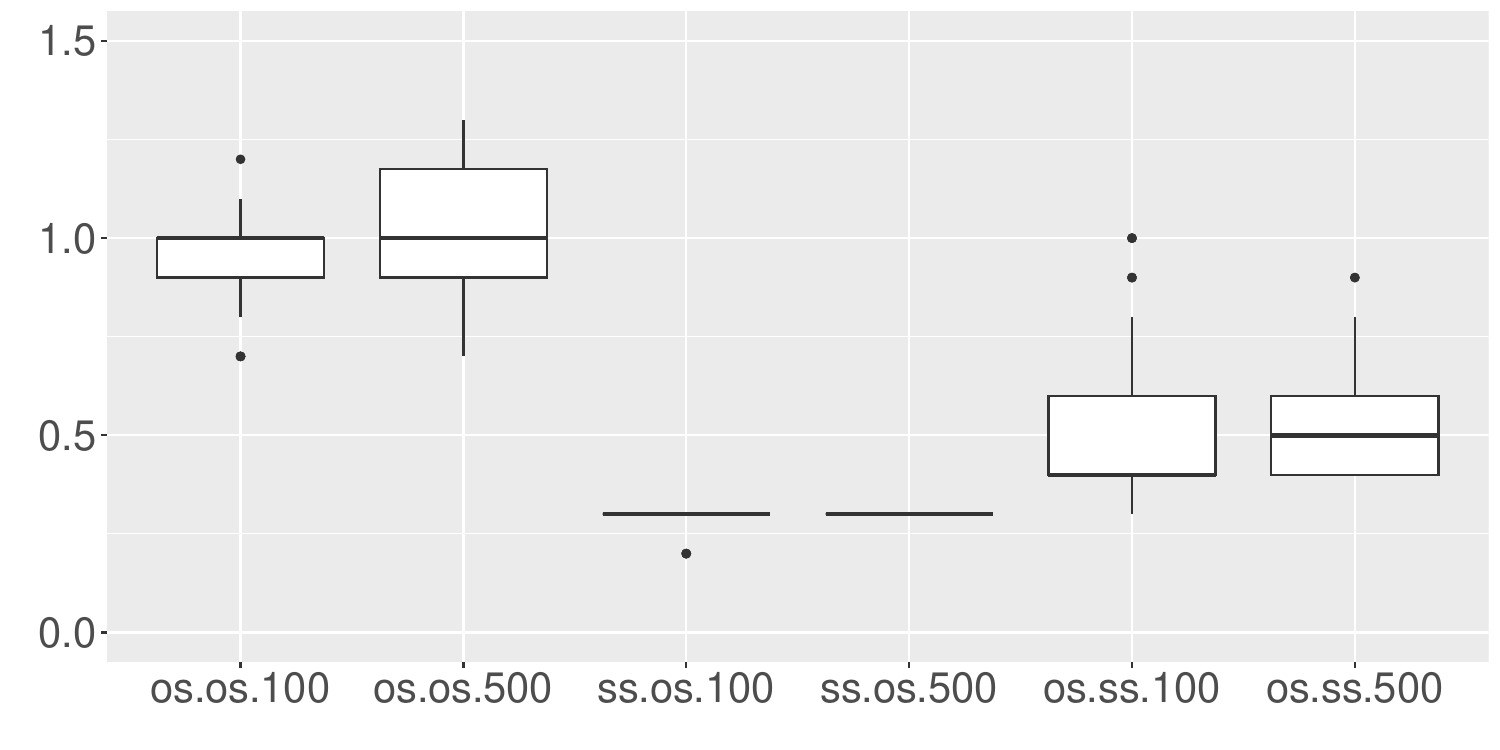}
    \caption{Values of the fully data-driven $\kestf$ for Examples  (o.s. - o.s.), (s.s. - o.s.) and (o.s - s.s.) over 50 iterations.}
    \label{fig:kest}
\end{figure}

\begin{enumerate}
    \item Ordinary Smooth - Ordinary Smooth (o.s. - o.s.): For $\denX = \denX_1$ and $\denU = \denU_1$ the true parameter is given by  $\wqf = \Vert \denX_1 \Vert_{\Lp[2]{+}} = 4/3 \approx 1.33$.
    \item Super Smooth - Ordinary Smooth (s.s. - o.s.): For $\denX = \denX_2$ and $\denU = \denU_1$ the true parameter is given by $\wqf= \Vert \denX_2 \Vert_{\Lp[2]{+}} = \sqrt[4]{\exp(1)}/(2\sqrt{\pi}) \approx 0.3622$.
    \item Ordinary Smooth - Super Smooth (o.s. - s.s.): For $\denX = \denX_1$ and $\denU = \denU_2$ the true parameter is again given by  $\wqf =  4/3 \approx 1.33$.
\end{enumerate}
In \cref{fig::estimation-different-dimensions}  for all three examples boxplots of the values of $\eqfk{k}$ are depicted over 50 iterations for $k\in\{0.1,0.3,\ldots,1.5\}$ and $n\in\{100,500\}$. Whenever the values where negative, the estimator is set to zero. Consequently, if there is a large fluctuation, they are largely estimated to be zero. We note that the simulation results in these three cases reflect their corresponding rates of convergence $n^{-4/7}$, $n^{-1}$ and $(\log n)^{-1}$ and highlight the importance of a proper choice of the dimension parameter $k$. \cref{fig:kest} shows for the same examples the fully data-driven estimates $\kestf$, see  \cref{eq::dim-k-fully}, with $\kappa = 0.00001$ over the set $k\in\{0.1,0.2,\ldots,2\}$. Here, we stopped the data-driven procedure as soon as there is too much fluctuation, i.e. as soon as there appeared the first negative estimated value.
The results indicate that with increasing $n$ the data-driven choice of $k$ increases also, which is expected from the theoretical results. Further, for exponential decay in the error density results deteriorate.

\paragraph{Acknowledgement}
This work is funded by Deutsche Forschungsgemeinschaft (DFG, German Research Foundation) under Germany's Excellence Strategy EXC-2181/1-39090098 (the Heidelberg STRUCTURES Cluster of Excellence) and by the Department for Science and Technology of the Embassy of France (PROCOPE Mobility grant).

\appendix
\setcounter{subsection}{0}
\section*{Appendix}
\numberwithin{equation}{subsection}  
\renewcommand{\thesubsection}{\Alph{subsection}}
\renewcommand{\theco}{\Alph{subsection}.\arabic{co}}
\numberwithin{co}{subsection}
\renewcommand{\thelem}{\Alph{subsection}.\arabic{lem}}
\numberwithin{lem}{subsection}
\renewcommand{\therem}{\Alph{subsection}.\arabic{rem}}
\numberwithin{rem}{subsection}
\renewcommand{\thepr}{\Alph{subsection}.\arabic{pr}}
\numberwithin{pr}{subsection}

 \subsection{Auxiliary results}\label{app::auxiliary}
		
		\begin{lem}\label{lem::normineq}
		\begin{itemize}
			\item[(i)] Let $h_1, h_2\in \Lp[1]{+}(\basMSy{2c-2})$ for $c\in\Rz$. Then $h_1 * h_2 \in \Lp[1]{+}(\basMSy{2c-2})$ with $\Vert h_1 * h_2 \Vert_{{\Lp[1]{+}(\basMSy{2c-2})}} \leq \Vert h_1 \Vert_{{\Lp[1]{+}(\basMSy{2c-2})}} \Vert h_2 \Vert_{{\Lp[1]{+}(\basMSy{2c-2})}} $. If $h_1$ and $h_2$ are probability densities, equality holds.
			\item[(ii)] Let $h_1, h_2\in \Lp[1]{+}(\basMSy{c-1})$ and $h_2\in \Lp[2]{+}(\basMSy{2c-1})$ for a $c\in\Rz$. Then $h_1 * h_2 \in \Lp[1]{+}(\basMSy{c-1}) \cap \Lp[2]{+}(\basMSy{2c-1})$ with $\Vert h_1 * h_2 \Vert_{{\Lp[2]{+}(\basMSy{2c-1})}} \leq \Vert h_1 \Vert_{{\Lp[1]{+}(\basMSy{c-1})}} \Vert h_2 \Vert_{{\Lp[2]{+}(\basMSy{2c-1})}} $.
			\item[(iii)] Let $h_1, h_2\in \Lp[\infty]{+}(\basMSy{2c-1})$ and $h_2 \in \Lp[1]{+}(\basMSy{2c-2})$ for  $c\in\Rz$. Then $h_1 * h_2 \in \Lp[\infty]{+}(\basMSy{2c-1})$ and it holds that $\Vert h_1 * h_2 \Vert_{{\Lp[\infty]{+}(\basMSy{2c-1})}} \leq \Vert h_1 \Vert_{{\Lp[\infty]{+}(\basMSy{2c-1})}} \Vert h_2 \Vert_{{\Lp[1]{+}(\basMSy{2c-2})}} $.
			\item[(iv)] If $f\in\Lp[1]{+}({\basMSy{2c-2}})$ and $\varphi\in\Lp[\infty]{+}( \basMSy{2c-1})\cap \Lp[1]{+}({\basMSy{2c-2}})$ are probability densities and $ g := f* \varphi$, we have that 
	\begin{align}\label{eq::mm3}
		\Vert g \Vert_{\Lp[\infty]{+}( \basMSy{2c-1})} \leq \Vert  f \Vert_{\Lp[1]{+}({\basMSy{2c-2}})} \Vert \varphi \Vert_{\Lp[\infty]{+}( \basMSy{2c-1})}  =  \frac{\Vert \varphi \Vert_{\Lp[\infty]{+}( \basMSy{2c-1})}}{\Vert  \varphi \Vert_{\Lp[1]{+}({\basMSy{2c-2}})}} \Vert  g \Vert_{\Lp[1]{+}({\basMSy{2c-2}})}.
	\end{align}
		\end{itemize}
	\end{lem}

	\begin{proof}[Lemma  \cref{lem::normineq}]
		Part (i) follows from Lemma 2.2.1  and (ii) from 2.3.1, \cite{BrennerMiguelDiss}.
		For part (iii) we see that
		\begin{align*}
			\Vert h_1 * h_2 \Vert_{{\Lp[\infty]{+}(\basMSy{2c-1})}} &= \sup_{y\in\pRz} \left\vert y^{2c-1} \int_{\pRz} h_1(y/x) h_2(x) x^{-1} dx \right\vert \\
			&\leq \int_{\pRz} \sup_{y\in\pRz} \vert y^{2c-1} h_1(y/x) h_2(x) x^{-1} \vert dx\\
			&=  \int_{\pRz} \vert h_2(x)\vert  x^{-1}  \sup_{y\in\pRz} \vert y^{2c-1} h_1(y/x) \vert dx\\
			&\leq  \int_{\pRz} \vert h_2(x)\vert  x^{-1} x^{2c-1} \sup_{y\in\pRz} \vert y^{2c-1} h_1(y) \vert dx\\
			&\leq \Vert h_2 \Vert_{{\Lp[1]{+}(\basMSy{2c-2})}} \Vert h_1 \Vert_{{\Lp[\infty]{+}(\basMSy{2c-1})}} 
		\end{align*}
		For the proof of (iv) apply (iii) and (i).
	\end{proof}
	For the Mellin transform we also have the following properties, which play key roles in the proofs.
	\begin{lem}\label{lem::Mellinprops}
	\begin{itemize}
	\item[(i)] It holds for any $h\in\Lp[1]{+}(\basMSy{c-1})\cap \Lp[2]{+}(\basMSy{2c-1})$ that 
	\begin{align*}
		\overline{\Mcg{c}{h}}(t)	= \Mcg{c}{h}(-t) , \quad \forall t\in\Rz
	\end{align*}
	and, analogously, for $H\in\Lp[1]{}\cap\Lp[2]{}$ for $x\in\pRz$  it holds $\overline{\Mcgi{c}{H}}(x) = \Mcgi{c}{H(-\cdot)}(x)$.
	\item[(ii)] Let $g\in\Lp[\infty]{+}(\basMSy{2c-1})$, then it holds for $\mathds{1}_{[-k,k]}h\in \Lp[2]{} $ that 
	\begin{align}
		\Ex[g] \left[ \left\vert\int_{-k}^k Y^{c-1+ 2\pi it} h(t)  dt \right\vert^2 \right] \leq  \Vert g \Vert_{\Lp[\infty]{+}( \basMSy{2c-1})} \Vert  \mathds{1}_{[-k,k]} h \Vert_{\Lp[2]{}}^2\label{eq::mm1}
	\end{align}
	\item[(iii)] For $g\in\Lp[\infty]{+}(\basMSy{2c-1})$ and $h\in\Lp[2]{+}(g)$ we have that
	\begin{align}\label{eq::mm2}
		\int_{\Rz} \left\vert \Ex[g] \left[ Y^{c-1+2\pi it} h(Y) \right]\right\vert^2 dt \leq \Vert g \Vert_{\Lp[\infty]{+}( \basMSy{2c-1})}  \Ex[g] [h^2(Y)].
	\end{align}
	\end{itemize}
\end{lem}
\begin{proof}[Proof of \cref{lem::Mellinprops}]
	Part (i) follows from integration by substitution. For (ii) we get with the Plancherel type identity, see  \cref{eq::Plancherel}, that
	\begin{align*}
	\Ex[g] \left[ \left\vert\int_{-k}^k Y^{c-1+ 2\pi it} h(t)  dt \right\vert^2 \right] &= \int_{\pRz} g(y) y^{2c-1} y^{1-2c} \left\vert \int_{-k}^k y^{c-1+ 2\pi it} h(t)  dt \right\vert^2 dy\nonumber\\
	&\leq  \Vert g \Vert_{\Lp[\infty]{+}( \basMSy{2c-1})} \int_{\pRz}y^{1-2c}\left\vert \Mcgi{1-c}{\mathds{1}_{[-k,k]}h} (y)\right\vert^2 dy \nonumber\\
	&=  \Vert g \Vert_{\Lp[\infty]{+}( \basMSy{2c-1})} \Vert \Mcgi{1-c}{\mathds{1}_{[-k,k]}h} \Vert_{\Lp[2]{+}(\basMSy{1-2c})}^2 \nonumber\\ 
	&=  \Vert g \Vert_{\Lp[\infty]{+}( \basMSy{2c-1})} \Vert  \mathds{1}_{[-k,k]} h \Vert_{\Lp[2]{}}^2.
\end{align*}
For (iii) we see that 
	\begin{align*}
	\int_{\Rz} \left\vert \Ex[g] \left[ Y^{c-1+2\pi it} h(Y) \right]\right\vert^2 dt &= \Vert \Mcg{c}{gh}\Vert_{\Lp[2]{}}  = \Vert gh\Vert_{\Lp[2]{+}({\basMSy{2c-1}})} \nonumber \\
	&= \int_{\pRz} g^2(y)h^2(y)y^{2c-1}dy \nonumber\leq \Vert g \Vert_{\Lp[\infty]{+}( \basMSy{2c-1})}  \E_g [h^2(Y)].
\end{align*}
\end{proof}
For the Proof of \cref{lem::upperbound-u-lin} we use the following adaption of Theorem 3, Chap.1,  from \cite{Lee1990}.

\begin{lem}\label{lem::u-stat-var}
	Let $U  := \frac{1}{n(n-1)} \sum_{\substack{j\not= l\\ j,l \in\llbracket n\rrbracket }}  h(X_{j}, X_{l})$ be a U-statistic with iid. random variables $(X_i)_{i\in\llbracket n\rrbracket}$. Then
	\begin{align*}
		\operatorname{Var}(U) \leq \frac{1}{n(n-1)} \left( 2(n-2) \Ex[][\vert \Ex[][h(X_1,X_2)]\vert^2] + \iEx[]{2}[\vert h(X_1,X_2)\vert^2] \right).
	\end{align*}
\end{lem}
The first bound in the next assertion, a concentration for canonical U-statistics, is a reformulation of Theorem 3.4.8 in \cite{Gine2016}.
\begin{lem}\label{lem::Ustatconcentration}
	Let $U_n$ be a U-statistic for $(X_j)_{j\in\llbracket n \rrbracket}$, $n\geq 2$ iid. $\pRz$-valued random variables and kernel $h\colon \pRz\times \pRz \rightarrow \Rz$ bounded and symmetric. Further, let $U_n$ be canonical, i.e. $\Ex[] [h(X_1,x)]= 0$ for all $x\in\pRz$. Assume that there are finite constants
	\begin{align*}
		&A\geq \Vert h\Vert_{\Lp[\infty]{+}(\Rz^2)},\\
		&B^2\geq \Vert \E[h^2(X_1, \cdot)] \Vert_{\Lp[\infty]{+}},\\
		&C^2 \geq\iEx[]{2}[h^2(X_1,X_2)],\\
		&D \geq  \sup \{\iEx[]{2}[h(X_1,X_2)\xi(X_1)\zeta(X_2)]: \E[\xi^2(X_1)]\leq 1, \E[\zeta^2(X_2)]\leq 1\}.
	\end{align*}
	Then, for all $x\geq 0$ it holds
	\begin{align*}
		\ipM[]{n} \left(U_n \geq 8 \frac Cn x^{1/2} + 13 \frac Dn x + 261\frac{B}{n^{3/2}}x^{\frac32} + 343 \frac{A}{n^2} x^2\right) \leq \exp(1- x).
	\end{align*}
	Hence, there exist universal numerical constants $\mathcal{C}, \kappa^*$ and $d$  such that it also holds for any $K\geq 1$
	\begin{align*}
		&\iEx[]{n}\left[\left(\vert U_n \vert^2 - (\log K) \left(\frac{\kappa^*C^2}{2n^2}+ \frac{\kappa^* (\log K)D^2}{32n^2} +\frac{d(\log K)^2 B^2}{4n^3} + \frac{d^2(\log K)^3 A^2}{16n^4}\right)\right)_+ \right] \\
		 &\leq \frac{\mathcal{C}}{K} \left( \frac{C^2}{n^2} + \frac{D^2}{n^2} + \frac{ B^2}{n^3}+ \frac{ A^2}{n^4}\right).
	\end{align*}
\end{lem}
\begin{proof}[Proof of \cref{lem::Ustatconcentration}]
	From the first bound, we obtain the second by integrating and choosing parameters $\kappa^*\geq 173.056$ and $d \geq 4.359.744$.
\end{proof}
The Bernstein inequality in the first formulation is for example given in \cite{Comte2017}, Appendix B, Lemma B.2. From the first bound, we obtain again the second by integrating.
\begin{lem}[Bernstein inequality]\label{lem::Bernstein-inequality}
		Let $(Z_j)_{j\in\llbracket n\rrbracket}$ be independent random variables satisfying $\E [Z_j]=0$, $\vert Z_j \vert\leq b$ almost surely and $\E[\vert Z_j\vert^2]\leq v$ for all $j\in\llbracket n\rrbracket$. Then, for all $x>0$ and $n\geq 1$, we have
	\begin{align*}
		\prob{} \left(\left\vert  \frac1n \sum_{j\in\llbracket n\rrbracket} Z_j \right\vert \geq x \right) \leq 2 \max \left( \exp\left(\frac{-nx^2}{4v}\right), \exp\left(\frac{-nx}{4b}\right) \right).
	\end{align*}
	Moreover, for any $K\geq 1$ we have
	\begin{align*}
		\E \left( \Big( \Big\vert  \frac1{\sqrt{n}} \sum_{j\in\llbracket n\rrbracket} Z_j \Big\vert^2 - (4v + 32b^2\log(K)n^{-1}) \log(K) \Big)_+\right) \leq \frac{8(v + 16b^2n^{-1}) }{K}. 
	\end{align*}
\end{lem} 
 \subsection{Proofs of \cref{sec::Adaptive}}\label{app::section4}

\subsubsection*{U-statistic results}

\begin{lem}[Concentration of the U-statistic]\label{lem::concentration-ustat}
	Under \cref{ass:well-definedness} there exists a universal numerical constant $\mathcal{C}>0$ such that
	\begin{align*}
		\iEx[\denY]{n} \left[ \max_{k \in\nset{M}}\left( \vert U_{k}\vert^2 -   \frac18 V(k)\right)_+ \right] \leq\mathcal{C} \icstC[\denU]{3} \icstV[\denX|\denU]{3} (1 \vee \Ex[\denY][Y_1^{4(c-1)}] )\frac{1}{n},
	\end{align*}
	for $U_k$ defined in  \cref{eq::u-stat} and $V(k)$ in  \cref{eq::pen} and $\kappa \geq 16^2 \kappa^*$ for $\kappa^*$ from \cref{lem::Ustatconcentration}.
\end{lem} 
\begin{proof}[Proof of \cref{lem::concentration-ustat}]
	For 
	$k\in\Nz$ consider the canonical U-statistic $U_k$  defined in  \cref{eq::u-stat} admitting as kernel the symmetric and real-valued function $h_k$ given in  \cref{eq::kernel-ustat} (see proof of \cref{lem::upperbound-u-lin}). We apply the concentration inequality for canonical U-statistics given in \cref{lem::Ustatconcentration}. Note that $\vert x^{c-1+2\pi it}\vert$ is not bounded for $x\in\pRz$ and hence $\vert h_k(x,y)\vert$ is generally not bounded for $x,y\in\pRz$. Therefore, we decompose $h_k$ in bounded and a remaining unbounded part. More precisely, given $\delta_k\in\pRz$, we denote for $y\in\pRz$ and $t\in\Rz$
	\begin{align}\label{eq::help-function-psi}
	\bcut{k}(y,t):= \mathds{1}_{[0,\delta_k]}(y^{c-1}) y^{c-1+2\pi it} \ \text{ and } \ \ucut{k}(y,t ):= \mathds{1}_{(\delta_k,\infty)}(y^{c-1}) y^{c-1+2\pi it}.
	\end{align}
	Define the bounded part of kernel $h_k$ as
	\begin{align}\label{eq::bounded-h-kernel}
		h^b_k(x, y) := \int_{-k}^{k} 				\frac{(\bcut{k}(y,t)-\Ex[\denY][\bcut{k}(Y_1,t)])(\bcut{k}(x,-t)-\Ex[\denY][\bcut{k}(Y_1,-t)])}{\vert \Mcg{c}{\denU}(t)\vert^2} \iwF{2}(t) dt.
	\end{align}
	Then, $h^b_k$ is indeed bounded since $\vert \bcut{k}(y,t)\vert \leq \delta_k$. Analogously, define
	\begin{align}\label{eq::unbounded-h-kernel}
		h^u_k(x, y) := \int_{-k}^{k} 				\frac{(\ucut{k}(y,t)-\Ex[\denY][\ucut{k}(Y_1,t)])(\ucut{k}(x,-t)-\Ex[\denY][\ucut{k}(Y_1,-t)])}{\vert \Mcg{c}{\denU}(t)\vert^2} \iwF{2}(t) dt
	\end{align}
	and
	 \begin{align}
	 	h^{bu}_k(x, y) :=& \int_{-k}^{k} 				\frac{(\ucut{k}(y,t)-\Ex[\denY][\ucut{k}(Y_1,t)])(\bcut{k}(x,-t)-\Ex[\denY][\bcut{k}(Y_1,-t)])}{\vert \Mcg{c}{\denU}(t)\vert^2} \iwF{2}(t) dt \nonumber \\
	 	& + \int_{-k}^{k} 				\frac{(\bcut{k}(y,t)-\Ex[\denY][\bcut{k}(Y_1,t)])(\ucut{k}(x,-t)-\Ex[\denY][\ucut{k}(Y_1,-t)])}{\vert \Mcg{c}{\denU}(t)\vert^2} \iwF{2}(t) dt. \label{eq::mixed-h-kernel}
	 \end{align}
 	Then, $h^b_k,h^{u}_k$ and $h^{bu}_k$ are also symmetric and real-valued and we have $h_k= h^b_k + h^u_k + h^{bu}_k$. Denote by $U^b_k$, $U^{u}_k$ and $U_k^{bu}$ the corresponding canonical U-statistics, respectively. Then, it holds $U_k = U^b_k + U^{u}_k + U_k^{bu}$. Consequently, we obtain
	\begin{align}
		&\iEx[\denY]{n} \left[ \max_{k\in\nset{M}}\left( \vert U_{k}\vert^2 -  \frac18 V(k)\right)_+ \right]\nonumber \\
		&\leq 2\sum_{k\in \nset{M}} \iEx[\denY]{n} \left[ \left( \vert U_{k}^b \vert^2 -  \frac1{16} V(k)\right)_+ \right]
		 + 4\sum_{k\in \nset{M}} \left(\iEx[\denY]{n}[\vert U^{u}_{k}\vert^2] + \iEx[\denY]{n}[\vert U^{bu}_{k}\vert^2] \right).\label{eq::U-stat-decomp}
	\end{align}
	We begin by considering the first summand. In \cref{lem::concentration-bounded-ustat} we apply the exponential inequality \cref{lem::Ustatconcentration} to $U_k^b$. To be more precise, with $\rate{k}$ defined in  \cref{eq::rate-k-omega} and constants  $\kappa^*, d>0 $  as in \cref{lem::Ustatconcentration}  we introduce
	\begin{align}\label{eq::K-and-deltak}
	  \delta_k^2:= \frac{\rate{k}}{32 d \icstC[\denU]{2} \icstV[\denX|\denU]{2}}\quad \text{ and } \quad K:= \rate{k}^4,
	\end{align}
	and 
	\begin{align*}
		V_1^b(k) := &\log(K)\icstC[\denU]{2} \icstV[\denX|\denU]{2}\frac{1}{n^2} \left(\Vert \mathds{1}_{[-k,k]} / \Mcg{c}{\denU} \Vert_{\Lp[4]{}(\iwF{4})}^4 \vee \log(K) \Vert \mathds{1}_{[-k,k]} / \Mcg{c}{\denU} \Vert_{\Lp[\infty]{}(\wF)}^4 \right) \\
		&\quad \cdot\left(\kappa^*  + d\frac{\delta_k^2}{n} (\log(K))^2+ d^2\frac{\delta_k^4 (2k)}{n^2}(\log(K))^3  \right).
	\end{align*}
	 We show below in  \cref{eq::v1b-bound} that $V_1^b(k)\leq  \frac1{16} V(k)$. 
	 With this we get from \cref{lem::concentration-bounded-ustat} that
	\begin{align*}
		&\iEx[\denY]{n} \left[\left( \vert U^b_{k}\vert^2 -  \frac1{16} V(k)\right)_+ \right]\leq 	\iEx[\denY]{n} \left[ \left( \vert U^b_{k}\vert^2 -  V_1^b(k)\right)_+ \right]\\
		&\leq \frac{\mathcal{C}\icstC[\denU]{2} \icstV[\denX|\denU]{2} }{K }\frac1{n^2}\left(\Vert \mathds{1}_{[-k,k]} / \Mcg{c}{\denU} \Vert_{\Lp[4]{}(\iwF{4})}^4 \vee  \Vert \mathds{1}_{[-k,k]} / \Mcg{c}{\denU} \Vert_{\Lp[\infty]{}(\wF)}^4 \right) \\
		&\qquad\qquad \cdot  \left(1 + \frac{\delta_k^2}{n} + \frac{\delta_k^4 (2k)}{n^2}  \right).
	\end{align*}
	Since by definition of $\rate{k}$  \cref{eq::rate-k-omega} we have
	\begin{align}
		\frac1n \left( \Vert \mathds{1}_{[-k,k]} / \Mcg{c}{\denU} \Vert_{\Lp[4]{}(\iwF{4})}^4 \vee  \Vert \mathds{1}_{[-k,k]} / \Mcg{c}{\denU} \Vert_{\Lp[\infty]{}(\wF)}^4 \right)\leq \frac{\rate{k}}{k^2}, \label{eq::omega-bound}
	\end{align}
	$\rate{k}\geq 1$, by  \cref{eq::K-and-deltak} $\frac{\delta_k^2}{\rate{k}}\leq 1$ and $\frac{2k}{\rate{k}}\leq 1$ and
	\begin{align*}
		\left(1 + \frac{\delta_k^2}{n} + \frac{\delta_k^4 (2k)}{n^2}  \right) \leq 3 \rate{k}^3.
	\end{align*}
	Combining the last inequalities,  \cref{eq::K-and-deltak} and $\sum_{k\in\Nz}k^{-2}=\pi^2/6$ we get
	\begin{align}
		&\sum_{k\in \nset{M}} \iEx[\denY]{n} \left[\left( \vert U^b_{k}\vert^2 -  \frac1{16} V(k)\right)_+ \right] \leq \sum_{k\in \nset{M}} \frac{\mathcal{C}\icstC[\denU]{2} \icstV[\denX|\denU]{2}\rate{k}^4 }{K k^2 }\frac1{n} \leq \frac{\mathcal{C}\icstC[\denU]{2} \icstV[\denX|\denU]{2}}{n} \label{eq::ustat-bounds}
	\end{align}
	To conclude the calculations for the first summand of  \cref{eq::U-stat-decomp}, we show now that $V_1^b(k)\leq  \frac1{16} V(k)$. Note that $K \geq 1$ for all $k\in\Nz$ and $\delta_k$ is increasing for $k\rightarrow \infty$. Further, we get
	\begin{align*}
		&\icstC[\denU]{2} \icstV[\denX|\denU]{2}\left(\kappa^*  + d\frac{\delta_k^2}{n} (\log(K))^2+ d^2 \frac{\delta_k^4 (2k)}{n^2}(\log(K))^3  \right)\\
		& \leq \kappa^* \icstC[\denU]{2} \icstV[\denX|\denU]{2} + \frac{\rate{k}(\log(\rate{k}))^2}{2n} + \frac{\rate{k}^2 (2k)(\log(\rate{k}))^3}{2n^2}\\
		&\leq \kappa^* \icstC[\denU]{2} \icstV[\denX|\denU]{2} + \left\vert1 \vee \frac{\rate{k} \log(\rate{k}) ((2k)\vee \log(\rate{k}))}{n} \right\vert^2
	\end{align*}
	Since  $\kappa\geq 16^2\kappa^* $ it follows 
	\begin{align}
		V_1^b(k) &\leq 4\log(\rate{k}) \frac{1}{n^2} \left(\Vert \mathds{1}_{[-k,k]} / \Mcg{c}{\denU} \Vert_{\Lp[4]{}(\iwF{4})}^4 \vee 4\log(\rate{k}) \Vert \mathds{1}_{[-k,k]} / \Mcg{c}{\denU} \Vert_{\Lp[\infty]{}(\wF)}^4 \right)\nonumber\\
		&\cdot \left(	\kappa^* \icstC[\denU]{2} \icstV[\denX|\denU]{2} + \left\vert1 \vee \frac{\rate{k} \log(\rate{k}) ((2k)\vee \log(\rate{k}))}{n} \right\vert^2\right) \nonumber\\
		& \leq  16\kappa^* /\kappa V(k)\leq \frac{1}{16} V(k).\label{eq::v1b-bound}
	\end{align} 
	Consider now the second summand of  \cref{eq::U-stat-decomp} which we bound with the help of  \cref{lem::concentration-unbounded-ustat} below. Its proof is based on
	the variance bound for U-statistics \cref{lem::u-stat-var} similarly as in \cref{lem::upperbound-u-lin}.  
	From \cref{lem::concentration-unbounded-ustat} it follows
	\begin{align*}
		&\sum_{k\in \nset{M}} (\iEx[\denY]{n}[\vert U^{u}_{k}\vert^2] + \iEx[\denY]{n}[\vert U^{bu}_{k}\vert^2]) \leq  \sum_{k\in \nset{M}} \frac{10\cstC[\denU] \cstV[\denX|\denU]}{n^2} \Vert \mathds{1}_{[-k,k]} / \Mcg{c}{\denU} \Vert_{\Lp[4]{}(\iwF{4})}^4 \delta_k^{-2} \Ex[\denY][Y_1^{4(c-1)}].
	\end{align*}
	Combining the last bound,  \cref{eq::omega-bound} and  \cref{eq::K-and-deltak} it follows that there exists a universal numerical constant $\mathcal{C}$ such that
	\begin{align*}
		\sum_{k\in \nset{M}} &(\iEx[\denY]{n}[\vert U^{u}_{k}\vert^2] + \iEx[\denY]{n}[\vert U^{bu}_{k}\vert^2])
		\leq 10\cstC[\denU] \cstV[\denX|\denU] \Ex[\denY][Y_1^{4(c-1)}]\frac{1}{n}\sum_{k\in \nset{M}} \frac{ \rate{k}}{k^2 \delta_k^2} \\
	 	&\leq \mathcal{C} \icstC[\denU]{3} \icstV[\denX|\denU]{3}\Ex[\denY][Y_1^{4(c-1)}]\frac{ 1 }{n}\sum_{k\in \Nz}\frac{1}{k^2} \leq \mathcal{C} \icstC[\denU]{3} \icstV[\denX|\denU]{3}\Ex[\denY][Y_1^{4(c-1)}]\frac{ 1 }{n}.
	\end{align*}
	Inserting the last bound and  \cref{eq::ustat-bounds} in  \cref{eq::U-stat-decomp} implies the result.
\end{proof}

\begin{lem}[Concentration of the bounded U-statistic]\label{lem::concentration-bounded-ustat}
	For $k\in\Nz$ and $\delta_k\in\pRz$ the real-valued, bounded and symmetric kernel $h^b_k\colon\Rz^2 \rightarrow \Rz$ given in  \cref{eq::bounded-h-kernel}  fulfills $\Ex[\denY] [h^b_k(Y_1, y)]= 0$ for all $y\in\pRz$. Under \cref{ass:well-definedness} for the canonical U-statistic
	\begin{align*}
		U^b_k:= \frac{1}{n(n-1)} \sum_{\substack{j\not= l\\ j,l \in\llbracket n\rrbracket }}  h^b_k(Y_j, Y_l)
	\end{align*}
	there exists a universal numerical constant $\mathcal{C}$ such that for each $K\geq1$
	\begin{align*}
	&\iEx[\denY]{n} \left[ \left( \vert U_{k}^b \vert^2 -  V_1^b(k)\right)_+ \right] \\
	&\quad  \leq \frac{\mathcal{C} \icstC[\denU]{2} \icstV[\denX|\denU]{2} }{K }\frac1{n^2}\left(\Vert \mathds{1}_{[-k,k]} / \Mcg{c}{\denU} \Vert_{\Lp[4]{}(\iwF{4})}^4 \vee  \Vert \mathds{1}_{[-k,k]} / \Mcg{c}{\denU} \Vert_{\Lp[\infty]{}(\wF)}^4 \right)\\
	&\qquad \qquad \cdot \left(1 + \frac{\delta_k^2}{n} + \frac{\delta_k^4 (2k)}{n^2}  \right),
	\end{align*}
	with  $ \cstC[\denU], \cstV[\denX|\denU]$ defined in  \cref{eq::constants}, numerical constants $\kappa^*, d>0$ given in  \cref{lem::Ustatconcentration} and
	\begin{align*}
		V_1^b(k) := &(\log K) \icstC[\denU]{2} \icstV[\denX|\denU]{2}\frac{1}{n^2} \left(\Vert \mathds{1}_{[-k,k]} / \Mcg{c}{\denU} \Vert_{\Lp[4]{}(\iwF{4})}^4 \vee (\log K) \Vert \mathds{1}_{[-k,k]} / \Mcg{c}{\denU} \Vert_{\Lp[\infty]{}(\wF)}^4 \right) \\
		&\quad \cdot\left(\kappa^* + d\frac{\delta_k^2}{n} (\log K)^2+ d^2\frac{\delta_k^4 (2k)}{n^2}(\log K)^3  \right).
	\end{align*}
\end{lem}

\begin{proof}[Proof of \cref{lem::concentration-bounded-ustat}]
	We 
	 intend to apply the concentration inequality for canonical U-statistics given in \cref{lem::Ustatconcentration} using the constants $\kappa^*$ and $d$ and the quantities $A$, $B$, $C$ and $D$ computed in \cref{lem::constants-bounded-ustat}.
	Using that $\Vert \mathds{1}_{[-k,k]} / \Mcg{c}{\denU} \Vert_{\Lp[2]{}(\iwF{2})}^4 \leq (2k) \Vert \mathds{1}_{[-k,k]} / \Mcg{c}{\denU} \Vert_{\Lp[\infty]{}(\iwF{2})}^4$ we get 
	\begin{align*}
		&\frac{\kappa^*C^2}{2n^2}+ \frac{\kappa^*(\log K)D^2}{8n^2} +\frac{d(\log K)^2 B^2}{4n^3} + \frac{d^2(\log K)^3 A^2}{16n^4}\\
		&\leq \kappa^* \frac{ \icstC[\denU]{2} \icstV[\denX|\denU]{2}}{n^2}\left(\Vert \mathds{1}_{[-k,k]} / \Mcg{c}{\denU} \Vert_{\Lp[4]{}(\iwF{4})}^4 \vee (\log K) \Vert \mathds{1}_{[-k,k]} / \Mcg{c}{\denU} \Vert_{\Lp[\infty]{}(\wF)}^4 \right)\\
		& \qquad + d\frac{(\log K )^2\cstC[\denU] \cstV[\denX|\denU]\delta_k^2 }{n^3} \Vert \mathds{1}_{[-k,k]} / \Mcg{c}{\denU} \Vert_{\Lp[4]{}(\iwF{4})}^4 \\
		&\qquad \quad+ d^2 \frac{(\log K)^3\delta_k^4  }{n^4} \Vert \mathds{1}_{[-k,k]} / \Mcg{c}{\denU} \Vert_{\Lp[2]{}(\iwF{2})}^4\\
		&\leq  \frac{ \icstC[\denU]{2} \icstV[\denX|\denU]{2}}{n^2}\left(\Vert \mathds{1}_{[-k,k]} / \Mcg{c}{\denU} \Vert_{\Lp[4]{}(\iwF{4})}^4 \vee (\log K) \Vert \mathds{1}_{[-k,k]} / \Mcg{c}{\denU} \Vert_{\Lp[\infty]{}(\wF)}^4 \right)\\
		&\qquad \cdot  \left(\kappa^* + d \frac{\delta_k^2}{n} (\log K)^2+ d^2\frac{\delta_k^4 (2k)}{n^2}(\log K)^3  \right).
	\end{align*}
	Consequently by \cref{lem::Ustatconcentration} there exists a universal numerical constant $\mathcal{C}$ such that
	\begin{align*}
		&\iEx[\denY]{n} \left[ \left( \vert U_{k}^b \vert^2 -  V_1^b(k)\right)_+ \right]\\
		&\leq \iEx[\denY]{n}\left[ \left(\vert U_{k}^b\vert^2 - (\log K) \left(\frac{\kappa^*C^2}{2n^2}+ \frac{\kappa^*(\log K)D^2}{8n^2} +\frac{d(\log K)^2 B^2}{4n^3} \right.\right.\right.\\
		&\hspace{9.5cm} \left.\left.\left.+ \frac{d^2(\log K)^3 A^2}{16n^4} \right)\right)_+ \right] \\
		 &\leq \frac{\mathcal{C}}{K} \left( \frac{C^2}{n^2} + \frac{D^2}{n^2} + \frac{ B^2}{n^3}+ \frac{ A^2}{n^4}\right)\\
		 &\leq \mathcal{C} \frac{\icstC[\denU]{2} \icstV[\denX|\denU]{2} }{K n^2}\left(\Vert \mathds{1}_{[-k,k]} / \Mcg{c}{\denU} \Vert_{\Lp[4]{}(\iwF{4})}^4 \vee  \Vert \mathds{1}_{[-k,k]} / \Mcg{c}{\denU} \Vert_{\Lp[\infty]{}(\wF)}^4 \right) \left(1 + \frac{\delta_k^2}{n} + \frac{\delta_k^4 (2k)}{n^2}  \right).
	\end{align*}
	which shows the result. 
\end{proof}

\begin{lem}[Constants for the bounded U-statistic]\label{lem::constants-bounded-ustat}
	For $k\in\Nz$ and $\delta_k\in\pRz$ the real-valued, bounded and symmetric kernel $h^b_k\colon\Rz^2 \rightarrow \Rz$ given in  \cref{eq::bounded-h-kernel}  fulfills $\Ex[\denY] [h^b_k(Y_1, y)]= 0$ for all $y\in\pRz$. Under \cref{ass:well-definedness} the following quantities satisfy the conditions in \cref{lem::Ustatconcentration}:
	\begin{align*}
		A &:= 4 \delta_k^2 \Vert \mathds{1}_{[-k,k]} / \Mcg{c}{\denU} \Vert_{\Lp[2]{}(\iwF{2})}^2,\\
		B^2 &:= 4\cstC[\denU] \cstV[\denX|\denU]\delta_k^2 \Vert \mathds{1}_{[-k,k]} / \Mcg{c}{\denU} \Vert_{\Lp[4]{}(\iwF{4})}^4,\\
		C^2 &:=\cstC[\denU] \icstV[\denX|\denU]{2}\Vert \mathds{1}_{[-k,k]} / \Mcg{c}{\denU} \Vert_{\Lp[4]{}(\iwF{4})}^4, \\
		D &:= 4\cstC[\denU] \cstV[\denX|\denU] \Vert \mathds{1}_{[-k,k]} / \Mcg{c}{\denU} \Vert_{\Lp[\infty]{}(\wF)}^2.
	\end{align*}
	
\end{lem}

\begin{proof}[Proof of \cref{lem::constants-bounded-ustat}]
	We compute the quantities $A$, $B$, $C$ and $D$ verifying the four inequalities of \cref{lem::Ustatconcentration}. Consider $A$. Since $\vert \bcut{k}(y,t) - \Ex[\denY][\bcut{k}(Y_1, t)]\vert \leq 2\delta_k$ for all $t\in\Rz$ and $y\in\pRz$ we have
	\begin{align*}
		\sup_{x,y\in\Rz} \vert h^b_k(x,y) \vert \leq 4 \delta_k^2 \int_{-k}^k \frac{1}{\vert\Mcg{c}{\denU}(t)\vert^2} \iwF{2}(t) dt  = A.
	\end{align*}
	Consider $B$. Using that the U-statistic is canonical and Assumption  \cref{eq::mm1} of \cref{lem::Mellinprops} we get for $y\in\pRz$ 
	\begin{align}
		&\Ex[\denY] [\vert h^b_k(y, Y_1)\vert^2] \leq \Ex[\denY] \left[\left\vert\int_{-k}^k  	\bcut{k}(Y_1,-t)\frac{(\bcut{k}(y,t)-\Ex[\denY][\bcut{k}(Y_1,t)])}{\vert \Mcg{c}{\denU}(t)\vert^2} \iwF{2}(t) dt \right\vert^2 \right]\nonumber\\
		&=  \Ex[\denY] \left[\mathds{1}_{[0, \delta_k]}(Y_1^{c-1}) \left\vert\int_{-k}^k Y_1^{c-1-2\pi it} \frac{(\bcut{k}(y,t)-\Ex[\denY][\bcut{k}(Y_1,t)])}{\vert \Mcg{c}{\denU}(t)\vert^2} \iwF{2}(t) dt \right\vert^2 \right] \nonumber\\
		&\leq \Ex[\denY] \left[ \left\vert\int_{-k}^k Y_1^{c-1-2\pi it} \frac{(\bcut{k}(y,t)-\Ex[\denY][\bcut{k}(Y_1,t)])}{\vert \Mcg{c}{\denU}(t)\vert^2} \iwF{2}(t) dt \right\vert^2 \right]\nonumber \\
		&\leq \Vert \denY  \Vert_{\Lp[\infty]{+}(\basMSy{2c-1})}  \int_{-k}^k \frac{\vert\bcut{k}(y,t)-\Ex[\denY][\bcut{k}(Y_1,t)]\vert^2}{\vert \Mcg{c}{\denU}(t)\vert^4} \iwF{4}(t) dt.  \label{eq::step-B-constant}
	\end{align}
	Consequently, $\vert \bcut{k}(y,t) - \Ex[\denY][\bcut{k}(Y_1, t)]\vert \leq 2\delta_k$ for all $t\in\Rz$  and $y\in\pRz$  and  \cref{eq::norm-bound} implies
	\begin{align*}
		\sup_{y\in\pRz} 	\Ex[\denY] [\vert h^b_k(y, Y_1)\vert^2] &\leq \cstC[\denU] \cstV[\denX|\denU] 4\delta_k^2  \int_{-k}^k \frac{1}{\vert \Mcg{c}{\denU}(t)\vert^4} \iwF{4}(t) dt = B^2. 
	\end{align*}
	Consider C. Keep in mind that for all $t\in\Rz$ 
	\begin{align*}
		&\Ex[\denY] \left[\vert \bcut{k}(Y_1,t) - \Ex[\denY][\bcut{k}(Y_1,t)]\vert^2\right] \leq \Ex[\denY]\left[\vert \bcut{k}(Y_1,t) \vert^2\right] \\
		&=  \Ex[\denY]\left[\mathds{1}_{[0,\delta_k]}(Y_1^{c-1})\vert Y_1^{c-1+2\pi it} \vert^2\right] \leq \Ex[\denY] [Y_1^{2c-2}] \leq \cstV[\denX|\denU].
	\end{align*}
	Using additionally the calculations of  \cref{eq::step-B-constant} for $B$ and  \cref{eq::norm-bound} again, it follows
	\begin{align*}
		\iEx[\denY]{2} \left[ \vert h^{b}_k(Y_1,Y_2)\vert^2  \right] &\leq  \Vert \denY  \Vert_{\Lp[\infty]{+}(\basMSy{2c-1})}  \int_{-k}^k \Ex[\denY][\vert\bcut{k}(y,t)-\Ex[\denY][\bcut{k}(Y_1,t)]\vert^2]\frac{\iwF{4}(t)}{\vert \Mcg{c}{\denU}(t)\vert^4}  dt\\
		&\leq \cstC[\denU] \icstV[\denX|\denU]{2} \int_{-k}^k \frac{1}{\vert \Mcg{c}{\denU}(t)\vert^4} \iwF{4}(t) dt = C^2.
	\end{align*}
	Finally, consider $D$. $D=C$ satisfies the condition, but we can find a sharper bound. For this, we first see that for all $\xi\in\Lp[2]{+}(\denY)$ due to Assumption  \cref{eq::mm2} in \cref{lem::Mellinprops} we have
	\begin{align*}
		\int_{\Rz} \left\vert \Ex[\denY] \left[\xi(Y_1)\bcut{k}(Y_1, t)\right] \right\vert^2 dt &= 
		\int_{\Rz} \left\vert \Ex[\denY] \left[Y_1^{c-1+2\pi it}\xi(Y_1)\mathds{1}_{[0, \delta_k]}(Y_1^{c-1}) \right] \right\vert^2 dt\\
		&\leq \Vert \denY  \Vert_{\Lp[\infty]{+}(\basMSy{2c-1})} \Ex[\denY] \left[\vert \xi(Y_1)\vert^2 \mathds{1}_{[0, \delta_k]}(Y_1^{c-1})\right]\\
		&\leq  \cstC[\denU] \cstV[\denX|\denU] \Ex[\denY] [\vert \xi(Y_1)\vert^2]
	\end{align*}
	where we used again  \cref{eq::norm-bound}.
	Similarly, by Assumption  \cref{eq::mm2}  we get 
	\begin{align*}
	\int_{\Rz} \vert \Ex[\denY] [Y_1^{c-1+2\pi it}\mathds{1}_{[0, \delta_k]}(Y_1^{c-1}) ] \vert^2 dt \leq \cstC[\denU] \cstV[\denX|\denU],
	\end{align*}
	which in turn implies
	\begin{align*}
		&\int_{\Rz} \vert \Ex[\denY]\left[\xi(Y_1)\Ex[\denY][\bcut{k}(Y_1, t)]\right] \vert^2 dt \leq \int_{\Rz} \Ex[\denY]\left[\vert \xi(Y_1)\vert^2\right]  \vert\Ex[\denY][\bcut{k}(Y_1, t)]\vert^2 dt\\
		&=   \Ex[\denY]\left[\vert \xi(Y_1)\vert^2\right] \int_{\Rz} \vert \Ex[\denY] [Y_1^{c-1+2\pi it}\mathds{1}_{[0, \delta_k]}(Y_1^{c-1}) ] \vert^2 dt\leq \Ex[\denY]\left[\vert \xi(Y_1)\vert^2\right] \cstC[\denU] \cstV[\denX|\denU].
	\end{align*}
	Combining the last bounds it follows
	 \begin{align*}
		\int_{\Rz} \vert \Ex[\denY][\xi(Y_1)(\bcut{k}(Y_1, t)-\Ex[\denY][\bcut{k}(Y_1,t)])]\vert^2 dt\leq 4\cstC[\denU] \cstV[\denX|\denU]\Ex[\denY][\vert\xi(Y_1)\vert^2].
	 \end{align*}
	  Consequently, for all $\xi,\zeta\in\Lp[2]{+}(\denY)$ with $\Ex[\denY][\vert \xi(Y_1)\vert^2]\leq 1$ and $\Ex[\denY][\vert \zeta(Y_1)\vert^2]\leq 1$ it follows
	\begin{align*}
		&\int_{\pRz} h^b_k(x,y)\xi(x)\zeta(y) \pM[\denY](dx)\pM[\denY](dy)\\
		& = \int_{-k}^k \text{\small$\frac{\Ex[\denY][\xi(Y_1)(\bcut{k}(Y_1,t)-\Ex[\denY][\bcut{k}(Y_1,t)])]\Ex[\denY][\zeta(Y_2)(\bcut{k}(Y_2,-t)-\Ex[\denY][\bcut{k}(Y_1, -t)])]}{\vert \Mcg{c}{\denU}(t)\vert^2}$} \iwF{2}(t)dt\\
		&\leq \sup_{\substack{\xi \in \Lp[2]{+}(\denY),\\ \Ex[\denY][\vert \xi(Y_1)\vert^2]\leq 1 }}\left\{ \int_{-k}^k \frac{\vert\Ex[\denY][\xi(Y_1)(\bcut{k}(Y_1,t)-\Ex[\denY][\bcut{k}(Y_1,t)])]\vert^2}{\vert \Mcg{c}{\denU}(t)\vert^2}\iwF{2}(t)dt\right\}\\
		&\leq \Vert \mathds{1}_{[-k,k]} / \Mcg{c}{\denU} \Vert_{\Lp[\infty]{}(\wF)}^2 \sup_{\substack{\xi \in \Lp[2]{+}(\denY),\\ \Ex[\denY][\vert \xi(Y_1)\vert^2]\leq 1 }}\left\{ \int_{-k}^k \vert\Ex[\denY][\xi(Y_1)(\bcut{k}(Y_1,t)-\Ex[\denY][\bcut{k}(Y_1,t)])]\vert^2 dt\right\}\\
		&\leq 4\cstC[\denU] \cstV[\denX|\denU] \Vert \mathds{1}_{[-k,k]} / \Mcg{c}{\denU} \Vert_{\Lp[\infty]{}(\wF)}^2 = D.
	\end{align*}
	This concludes the proof.
\end{proof}

\begin{lem}[Variance of the unbounded U-statistic]\label{lem::concentration-unbounded-ustat}
	For $k\in\Nz$ and $\delta_k\in\pRz$ the real-valued, bounded and symmetric kernels $h^{u}_k, h^{bu}_k\colon\pRz^2 \rightarrow \Rz$ given in  \cref{eq::unbounded-h-kernel} and  \cref{eq::mixed-h-kernel}, respectively, fulfill for all $y\in\pRz$ that $\Ex[\denY] [h^{u}_k(Y_1, y)]= 0$ and $\Ex[\denY] [h^{bu}_k(Y_1, y)]= 0$  . Under \cref{ass:well-definedness} the canonical U-statistics 
	\begin{align*}
		U^u_k := \frac{1}{n(n-1)} \sum_{\substack{j\not= l\\ j,l \in\llbracket n\rrbracket }}  h^u_k(Y_j, Y_l),\quad \text{ and } \quad 
		U^{bu}_k := \frac{1}{n(n-1)} \sum_{\substack{j\not= l\\ j,l \in\llbracket n\rrbracket }}  h^{bu}_k(Y_j, Y_l),
	\end{align*}
	satisfy that
	\begin{align*}
		\Ex[\denY] [\vert U_k^{u}\vert^2] &\leq \frac{2\cstC[\denU] \cstV[\denX|\denU]}{n^2} \Vert \mathds{1}_{[-k,k]} / \Mcg{c}{\denU} \Vert_{\Lp[4]{}(\iwF{4})}^4 \delta_k^{-2} \Ex[\denY][Y_1^{4(c-1)}], \\
		\Ex[\denY] [\vert U_k^{bu}\vert^2] &\leq \frac{8\cstC[\denU] \cstV[\denX|\denU]}{n^2} \Vert \mathds{1}_{[-k,k]} / \Mcg{c}{\denU} \Vert_{\Lp[4]{}(\iwF{4})}^4 \delta_k^{-2} \Ex[\denY][Y_1^{4(c-1)}].
	\end{align*}
\end{lem}

\begin{proof}[Proof of \cref{lem::concentration-unbounded-ustat}]
 Applying \cref{lem::u-stat-var} we get ${\iEx[\denY]{n} [\vert U_k^{u}\vert^2] \leq \frac{2}{n^2} \iEx[\denY]{2} [\vert h^{u}_k(Y_1,Y_2) \vert^2]}$. With \cref{lem::Mellinprops},  \cref{eq::mm1} we have analogously to the calculations in  \cref{eq::step-B-constant} in \cref{lem::constants-bounded-ustat} that
 \begin{align*}
	\Ex[\denY] [\vert h^{u}_k(y,Y_1) \vert^2]
	&\leq \Vert \denY  \Vert_{\Lp[\infty]{+}(\basMSy{2c-1})}  \int_{-k}^k \frac{\vert\ucut{k}(y,t)-\Ex[\denY][\ucut{k}(Y_1,t)]\vert^2}{\vert \Mcg{c}{\denU}(t)\vert^4} \iwF{4}(t) dt.  
 \end{align*}
Following the computations of the proof of \cref{lem::constants-bounded-ustat}, for all $t\in\Rz$ we have that 
 \begin{align*}
	\Ex[\denY] \left[\vert \ucut{k}(Y_1,t) - \Ex[\denY][\ucut{k}(Y_1,t)]\vert^2\right] &\leq \Ex[\denY]\left[\vert \ucut{k}(Y_1,t) \vert^2\right] \\
		&=  \Ex[\denY]\left[\mathds{1}_{(\delta_k, \infty)}(Y_1^{c-1}) Y_1^{2(c-1)} \right] \leq \delta_k^{-2} \Ex[\denY] [Y_1^{4(c-1)}].
 \end{align*}
 	Combining the last bounds and using  \cref{eq::norm-bound} yield
	\begin{align}
		\frac{n^2}{2 }\iEx[\denY]{n} [\vert U_k^{u}\vert^2] &\leq \iEx[\denY]{2} [\vert h^{u}_k(Y_1,Y_2) \vert^2]\nonumber\\
		&\leq \Vert \denY  \Vert_{\Lp[\infty]{+}(\basMSy{2c-1})}  \int_{-k}^k\Ex[\denY][\vert\ucut{k}(Y_1,t)-\Ex[\denY][\ucut{k}(Y_1,t)]\vert^2] \frac{\iwF{4}(t)}{\vert \Mcg{c}{\denU}(t)\vert^4}  dt\nonumber\\
		&\leq \cstC[\denU] \cstV[\denX|\denU]\delta_k^{-2} \Ex[\denY] [Y_1^{4(c-1)}]\int_{-k}^k \frac{1}{\vert \Mcg{c}{\denU}(t)\vert^4}  \iwF{4}(t)dt . \label{eq::upper-h}
	\end{align}
	Rearranging the last inequality implies the first result.  Applying \cref{lem::u-stat-var} we get that $\iEx[\denY]{n} [\vert U_k^{bu}\vert^2] \leq \frac{2}{n^2} \iEx[\denY]{2} [\vert h^{bu}_k(Y_1,Y_2) \vert^2]$. Now, set
	\begin{align*}
		\tilde{h}^{ub}_k(x,y) := \int_{-k}^{k}\frac{(\ucut{k}(y,t)-\Ex[\denY][\ucut{k}(Y_1,t)])(\bcut{k}(x,-t)-\Ex[\denY][\bcut{k}(Y_1,-t)])}{\vert \Mcg{c}{\denU}(t)\vert^2} \iwF{2}(t) dt 
	\end{align*}
	and $\tilde{h}^{bu}_k(x,y)$ similarly. Then, we have $h^{bu}_k(x,y):= \tilde{h}^{bu}_k(x,y) + \tilde{h}^{ub}_k(x,y) $ for any $x,y\in\pRz$. It follows that $\iEx[\denY]{n} [\vert U_k^{bu}\vert^2]\leq \frac{4}{n^2}\iEx[\denY]{2} [\vert \tilde{h}^{bu}_k(Y_1,Y_2) \vert^2] +\frac{4}{n^2}\iEx[\denY]{2} [\vert \tilde{h}^{ub}_k(Y_1,Y_2) \vert^2] $. 
	Similarly to  \cref{eq::step-B-constant} in \cref{lem::constants-bounded-ustat}  we get 
	\begin{align*}
		&\Ex[\denY] [\vert \tilde{h}^{bu}_k(Y_1,y) \vert^2]+ \Ex[\denY] [\vert \tilde{h}_k^{ub}(y,Y_1) \vert^2] \\
		&\leq 2\Vert \denY  \Vert_{\Lp[\infty]{+}(\basMSy{2c-1})}  \int_{-k}^k \frac{\vert\ucut{k}(y,t)-\Ex[\denY][\ucut{k}(Y_1,t)]\vert^2}{\vert \Mcg{c}{\denU}(t)\vert^4} \iwF{4}(t) dt.
	\end{align*}
	Analogously to  \cref{eq::upper-h} it follows
	\begin{align*}
		\frac{n^2}{4}\iEx[\denY]{n} [\vert U_k^{bu}\vert^2]  &\leq \Ex[\denY] [\vert \tilde{h}^{bu}_k(Y_1,y) \vert^2]+ \Ex[\denY] [\vert \tilde{h}_k^{ub}(y,Y_1) \vert^2]\\
		&\leq 2\cstC[\denU] \cstV[\denX|\denU]\delta_k^{-2} \Ex[\denY] [Y_1^{4(c-1)}]\int_{-k}^k \frac{1}{\vert \Mcg{c}{\denU}(t)\vert^4}  \iwF{4}(t)dt .
	\end{align*}
	Rearranging the last inequality implies the second result. 
\end{proof}

\subsubsection*{Linear statistic results}

\begin{lem}[Concentration of the linear statistic]\label{lem::concentration-linear}
	Under \cref{ass:well-definedness} there exists a universal numerical constant $\mathcal{C}$ such that 
	\begin{align*}
		&\iEx[\denY]{n} \left[ \max_{k <k'\leq M}\left( \vert W_{k'}- W_k\vert^2 -(\frac1{32} V(k') + \frac{1}{16}\bias{k}{2})\right)_+ \right]\\
		&\leq \mathcal{C}(\icstC[\denU]{2} \icstV[\denX|\denU]{2} +  (\Ex[\denY] [Y_1^{4(c-1)}])^2)\frac{1}{n} + \mathcal{C}\bias{k}{2}.
	\end{align*}
	for $W_k$ defined in  \cref{eq::linear-term}, $V(k)$ in  \cref{eq::pen} and $\kappa \geq 64^2*4*128^2$.
\end{lem}

\begin{proof}[Proof of \cref{lem::concentration-linear}]
	For $k\in\Nz$ consider $W_k$ defined in  \cref{eq::linear-term}. We intend to apply the concentration inequality \cref{lem::Bernstein-inequality} where we need to compute the quantities $b$ and $v$ verifying the required inequalities. Analogously to the U-statistic part, note that $W_{k'}-W_k$ is generally not bounded since $\vert x^{c-1+it}\vert$ is not bounded for $x\in\pRz$. 
	Therefore, we use again the notation $\bcut{k}$ and $\ucut{k}$ analogously to  \cref{eq::help-function-psi}, i.e. given a sequence $(\delta_k)_{k\in\Nz}$ define for $k\in\Nz$ for $y\in\pRz$ and $t\in\Rz$
	\begin{align*}
		\bcut{k}(y,t):= \mathds{1}_{[0,\delta_k]}(y^{c-1}) y^{c-1+2\pi it} \ \text{ and } \ \ucut{k}(y,t ):= \mathds{1}_{(\delta_k,\infty)}(y^{c-1}) y^{c-1+2\pi it}.
	\end{align*}
	Then, for $k,k'\in\nset{M}$ with $k\leq k'$ we define
	\begin{align}
		\linstatb{k'}{k} &:= \frac{1}{n} \sum_{j \in\llbracket n\rrbracket } \int_{[-k',k']\setminus [-k,k]} (\bcut{k'}(Y_j, t) -  \Ex[\denY][\bcut{k'}(Y_1, t)]) \frac{ \Mcg{c}{\denY}(-t)}{\vert \Mcg{c}{\denU}(t)\vert^2}  \iwF[]{2} (t)  dt, \label{eq::linear-statistic-bouded}\\
		\linstatu{k'}{k} &:= \frac{1}{n} \sum_{j \in\llbracket n\rrbracket } \int_{[-k',k']\setminus [-k,k]}(\ucut{k'} (Y_j, t) -  \Ex[\denY][\ucut{k'}(Y_1, t)])\frac{ \Mcg{c}{\denY}(-t)}{\vert \Mcg{c}{\denU}(t)\vert^2}  \iwF[]{2} (t)  dt.\label{eq::linear-statistic-unbouded}
	\end{align}
	We evidently have $Y_j^{c-1+2\pi it} = \bcut{k'}(Y_j,t) + \ucut{k'}(Y_j,t)$ for any $k'\in\nset{M}$ and, consequently, $W_{k'} - W_k = \linstatb{k'}{k} + \linstatu{k'}{k}$ for any $k\in\nset{M}$ with $k\leq k'$. Note that we choose the cut-off of the random variable $Y_j$ only in dependence of the larger dimension parameter $k'\in\nset{M}$.
	With this, we immediately obtain
	\begin{align}
		&\iEx[\denY]{n} \left[ \max_{k <k'\leq M}\left( \vert W_{k'}- W_k\vert^2 -   (\frac1{32} V(k') + \frac{1}{16}\bias{k}{2})\right)_+ \right]\nonumber\\
		&\leq 2 \sum_{k <k'\leq M}\iEx[\denY]{n} \left[ \left( \vert \linstatb{k'}{k}\vert^2 -   (\frac1{64} V(k') + \frac{1}{32}\bias{k}{2})\right)_+ \right] \nonumber\\
		&\qquad + 2 \sum_{k < k'\leq M} \iEx[\denY]{n} [\vert \linstatu{k'}{k}\vert^2].\label{eq::sums}
	\end{align}
	We begin by considering the first summand. In \cref{lem::concentration-bounded-lin} we apply the Bernstein inequality \cref{lem::Bernstein-inequality} to $\linstatb{k'}{k}$.
	To be more precise, with $\rate{k}$ defined in  \cref{eq::rate-k-omega} and constants  $\kappa^*, d>0 $  as in \cref{lem::Ustatconcentration}  we introduce
	\begin{align}
	  \delta_{k'}^2&:= \rate{k'} \quad \text{ and}
	   \quad K:= \rate{k'}^2 \label{eq::K-and-deltak2}.
	\end{align}
	Further, set 
	\begin{align*}
		V_2^b(k') &:=  (\log K) \Vert (\mathds{1}_{[-k',k']} - \mathds{1}_{[-k,k]})\Mcg{c}{\denX} / \Mcg{c}{\denU}  \Vert^2_{\Lp[2]{}(\iwF{4})} \\
		&\hspace{1.5cm}\cdot \frac{1}{n} \left(4\cstC[\denU] \cstV[\denX|\denU] + \frac{128\delta_{k'}^2 (2k') (\log K)}{n}\right).
	\end{align*}
	We show below in  \cref{eq::vb-inequ} that $V_2^b(k)\leq\frac1{64} V(k) + \frac{1}{32}\bias{k}{2} $. With this we get from \cref{lem::concentration-bounded-lin} that
	\begin{align*}
		&\iEx[\denY]{n} \left[ \left( \vert \linstatb{k'}{k}\vert^2 -   (\frac1{64} V(k') + \frac{1}{32}\bias{k}{2})\right)_+ \right]\leq \iEx[\denY]{n} \left[ \left( \vert \linstatb{k'}{k}\vert^2 -   V^b_2(k')\right)_+ \right]\\
		&\leq \frac{\mathcal{C}}{K} \Vert (\mathds{1}_{[-k',k']} - \mathds{1}_{[-k,k]})\Mcg{c}{\denX} / \Mcg{c}{\denU}  \Vert^2_{\Lp[2]{}(\iwF{4})}  \frac{1}{n}\left(\cstC[\denU] \cstV[\denX|\denU] + \frac{\delta_{k'}^2 (2k') }{n}\right).
	\end{align*}
	For $a,b\in\Rz$ it holds $2ab \leq  a^2 + b^2$
	\begin{align}
		&2\frac1{K}\Vert (\mathds{1}_{[-k',k']} - \mathds{1}_{[-k,k]})\Mcg{c}{\denX} / \Mcg{c}{\denU}  \Vert^2_{\Lp[2]{}(\iwF{4})} \frac{1}{n}\left(\cstC[\denU] \cstV[\denX|\denU] + \frac{\delta_{k'}^2 (2k') }{n}\right)\nonumber\\
		 &\leq 2 \frac1{K}\Vert \mathds{1}_{[-k',k']}/ \Mcg{c}{\denU} \Vert_{\Lp[\infty]{}(\wF)}^2  \Vert \mathds{1}_{\Rz\setminus[-k,k]}\Mcg{c}{\denX} \Vert_{\Lp[2]{}(\iwF{2})}^2\frac{1}{n}\left(\cstC[\denU] \cstV[\denX|\denU] + \frac{\delta_{k'}^2 (2k') }{n}\right)\nonumber\\
		&\leq \frac1{K^2}k'^{2}\Vert \mathds{1}_{[-k',k']}/ \Mcg{c}{\denU} \Vert_{\Lp[\infty]{}(\wF)}^4 \frac{1}{n^2}\left(\cstC[\denU] \cstV[\denX|\denU] + \frac{\delta_{k'}^2 (2k') }{n}\right)^2\nonumber\\
		&\hspace{1.5cm}+  \frac{1}{k'^2} \Vert \mathds{1}_{\Rz\setminus[-k,k]}\Mcg{c}{\denX} \Vert_{\Lp[2]{}(\iwF{2})}^4\nonumber\\
		&\leq \mathcal{C}\icstC[\denU]{2} \icstV[\denX|\denU]{2}\frac{1}{k'^2n}\frac{1}{K^2} \frac1n \Vert \mathds{1}_{[-k',k']}/ \Mcg{c}{\denU} \Vert_{\Lp[\infty]{}(\wF)}^4(2k')^6  \delta_{k'}^4 + \frac{1}{k'^2} \bias{k}{2}.\label{eq::splitting}
	\end{align}
	Since by definition of $\rate{k}$ in  \cref{eq::rate-k-omega} and by  \cref{eq::K-and-deltak2} we have that
	\begin{align}
		\frac1n \Vert \mathds{1}_{[-k',k']}/ \Mcg{c}{\denU} \Vert_{\Lp[\infty]{}(\wF)}^4   (2k')^6  \leq \rate{k'}^2 =\frac{K^2}{\delta_{k'}^4}  . \label{eq::omega-bound2}
	\end{align}
	 Combining the last inequalities and  since $\sum_{k'\in\Nz}k'^{-2} = \pi^2/6$  we get
	\begin{align}
		 &\sum_{k <k'\leq M}\iEx[\denY]{n} \left[ \left( \vert \linstatb{k'}{k}\vert^2 -   (\frac1{64} V(k') + \frac{1}{32}\bias{k}{2})\right)_+ \right]\nonumber\\
		  &\leq \mathcal{C}\icstC[\denU]{2} \icstV[\denX|\denU]{2}\frac{1}{n} \sum_{k'\in\Nz}\frac{1}{k'^2}  +\bias{k}{2}\sum_{k'\in\Nz} \frac{1}{k'^2} \nonumber\\
		&\leq \mathcal{C}\icstC[\denU]{2} \icstV[\denX|\denU]{2}\frac{1}{n} + \mathcal{C}\bias{k}{2}.\label{eq::bound-bounded-lin}
	\end{align}
	To conclude the calculation for the first summand of  \cref{eq::sums}, we show now that $V_2^b(k)\leq\frac1{64} V(k) + \frac{1}{32}\bias{k}{2} $. 
	For this, with similar calculations as in   \cref{eq::splitting} we see with  \cref{eq::K-and-deltak2} that
	 \begin{align}
		&V_2^b(k')\nonumber\\
		&= (\log K) \Vert (\mathds{1}_{[-k',k']} - \mathds{1}_{[-k,k]})\Mcg{c}{\denX} / \Mcg{c}{\denU}  \Vert^2_{\Lp[2]{}(\iwF{4})}\nonumber\\
		&\qquad \cdot\frac{1}{n} \left(4\cstC[\denU] \cstV[\denX|\denU] + \frac{128\delta_{k'}^2 (2k') (\log K)}{n}\right)\nonumber\\
		&\leq \frac{2}{n} \log(\rate{k'}) \Vert \mathds{1}_{[-k',k']}/ \Mcg{c}{\denU} \Vert_{\Lp[\infty]{}(\wF)}^2 \Vert (\mathds{1}_{[-k',k']} - \mathds{1}_{[-k,k]}) \Mcg{c}{\denX} \Vert_{\Lp[2]{}(\iwF{2})}^2 \nonumber\\
		&\qquad \cdot \left(4\cstC[\denU] \cstV[\denX|\denU] + \frac{128\rate{k'} (2k') 2(\log \rate{k'})}{n}\right)\nonumber\\
		&\leq 64 \frac{(4\cstC[\denU] \cstV[\denX|\denU])^2}{n^2} (\log \rate{k'})^2 \Vert \mathds{1}_{[-k',k']}/ \Mcg{c}{\denU} \Vert_{\Lp[\infty]{}(\wF)}^4\nonumber\\
		&\qquad +  \frac{1}{64}\Vert (\mathds{1}_{[-k',k']} - \mathds{1}_{[-k,k]}) \Mcg{c}{\denX} \Vert_{\Lp[2]{}(\iwF{2})}^4\nonumber\\
		&\qquad + 64\frac{1}{n^2}(\log \rate{k'})^2 \Vert \mathds{1}_{[-k',k']}/ \Mcg{c}{\denU} \Vert_{\Lp[\infty]{}(\wF)}^4 \frac{4*128^2 \rate{k'}^2 (2k')^2 (\log\rate{k'})^2}{n^2}\nonumber\\
		&\qquad + \frac{1}{64}\Vert (\mathds{1}_{[-k',k']} - \mathds{1}_{[-k,k]}) \Mcg{c}{\denX} \Vert_{\Lp[2]{}(\iwF{2})}^4\nonumber\\
		&\leq 64 \frac{(\log \rate{k'})^2}{n^2} \Vert \mathds{1}_{[-k',k']}/ \Mcg{c}{\denU} \Vert_{\Lp[\infty]{}(\wF)}^4 \left((4\cstC[\denU] \cstV[\denX|\denU])^2 + \frac{4*128^2\rate{k'}^2 (2k')^2 (\log\rate{k'})^2}{n^2} \right)\nonumber\\
		&\qquad + \frac{1}{32} \bias{k}{2}\nonumber.
	 \end{align}
	Exploiting the definition of $V(k')$, see  \cref{eq::pen}, it follows
	 \begin{align}
		V_2^b(k')&\leq \frac{64*4*128^2}{\kappa} V(k') + \frac{1}{32} \bias{k}{2} \leq \frac1{64} V(k') + \frac{1}{32} \bias{k}{2} \label{eq::vb-inequ}
	\end{align}
	since $\kappa \geq 64^2*4*128^2$. For the unbounded part, i.e. the second summand of  \cref{eq::sums}, we get with \cref{lem::concentration-unbounded-lin}, with similar calculations as in   \cref{eq::splitting}, \cref{eq::bound-bounded-lin} and using  \cref{eq::omega-bound2} that
	\begin{align*}
		&\sum_{k < k'\leq M} \Ex[\denY] [\vert \linstatu{k'}{k}\vert^2]\\
		&\leq  \sum_{k < k'\leq M}\frac{1}{n} \Vert (\mathds{1}_{[-k',k']} - \mathds{1}_{[-k,k]}) \Mcg{c}{\denX} / \Mcg{c}{\denU}  \Vert^2_{\Lp[2]{}(\iwF{4})} (2k') \delta_{k'}^{-2} \Ex[\denY] [Y_1^{4(c-1)}]\\
		&\leq \sum_{k'\in\Nz} \left((\Ex[\denY] [Y_1^{4(c-1)}])^2  \frac{k'^4}{\rate{k'}^2 n^2}\Vert \mathds{1}_{[-k',k']}/ \Mcg{c}{\denU} \Vert_{\Lp[\infty]{}(\wF)}^4  + \frac{1}{k'^2}\Vert \mathds{1}_{\Rz\setminus[-k,k]}\Mcg{c}{\denX} \Vert_{\Lp[2]{}(\iwF{2})}^4 \right)\\
		&\leq \mathcal{C} (\Ex[\denY] [Y_1^{4(c-1)}])^2 \frac{1}{n} + \mathcal{C}\bias{k}{2} .
	\end{align*}
	Inserting now the last bound and  \cref{eq::bound-bounded-lin} into  \cref{eq::sums} yields the result.
\end{proof}

\begin{lem}[Concentration of the bounded linear statistic]\label{lem::concentration-bounded-lin}
	Let the centered linear statistic $\linstatb{k'}{k}$ for  $k,k'\in\Nz$, $k'\geq k$  and $\delta_{k'}\in\pRz$, be defined as in  \cref{eq::linear-statistic-bouded}. Under \cref{ass:well-definedness} there exists a universal numerical constant $\mathcal{C}$ such that for each $K\geq1$
	\begin{align*}
		&\iEx[\denY]{n} \left[\left(  \vert \linstatb{k'}{k} \vert^2 -   V_2^b(k')\right)_+ \right] \\
		&\leq \frac{\mathcal{C}}{K} \Vert (\mathds{1}_{[-k',k']} - \mathds{1}_{[-k,k]})\Mcg{c}{\denX} / \Mcg{c}{\denU}  \Vert^2_{\Lp[2]{}(\iwF{4})}  \frac{1}{n}\left(\cstC[\denU] \cstV[\denX|\denU] + \frac{\delta_{k'}^2 (2k') }{n}\right)
	\end{align*}
	with  $ \cstC[\denU], \cstV[\denX|\denU]$ defined in  \cref{eq::constants} and 
	\begin{align*}
		&V_2^b(k'):=(\log K) \Vert (\mathds{1}_{[-k',k']} - \mathds{1}_{[-k,k]})\Mcg{c}{\denX} / \Mcg{c}{\denU}  \Vert^2_{\Lp[2]{}(\iwF{4})}  \\
		& \hspace{2cm}\cdot \frac{1}{n} \left(4\cstC[\denU] \cstV[\denX|\denU] + \frac{128\delta_{k'}^2 (2k') (\log K)}{n}\right).
	\end{align*}
\end{lem}

\begin{proof}[Proof of \cref{lem::concentration-bounded-lin}]
	In the following, we denote for $k\leq k'$ the notation $\inddif{k'}{k} := \mathds{1}_{[-k',k']} - \mathds{1}_{[-k,k]}$.
	First, 
	for $i\in\nset{n}$ define the iid. random variables
	\begin{align*}
		Z_{i} := \int_{\Rz} (\bcut{k'}(Y_{i}, t) -  \Ex[\denY][\bcut{k'}(Y_1, t)]) \frac{\inddif{k'}{k}(t)\Mcg{c}{\denX}(-t)}{ \Mcg{c}{\denU}(t)}  \iwF[]{2} (t)  dt
	\end{align*}
	where $\Ex[\denY][Z_{i}] = 0$ and $\linstatb{k'}{k}= \frac{1}{n}\sum_{i\in\nset{n}} Z_i$. We intend to apply \cref{lem::Bernstein-inequality} which needs quantities $\Ex[\denY][\vert Z_i\vert^2]\leq v$ and $\vert Z_i \vert \leq b$. Consider $b$ first. Since  $\vert\bcut{k'}(y,t) - \Ex[\denY][\bcut{k'}(Y_1, t)]\vert \leq 2\delta_{k'}$ for all $t\in\Rz$ and $y\in\pRz$ we have
	\begin{align*}
		\left\vert Z_i \right\vert^2 &\leq 4\delta_{k'}^2\left\vert \int_{\Rz} \inddif{k'}{k}(t) \frac{\vert \Mcg{c}{\denX}(t)\vert}{\vert\Mcg{c}{\denU}(t)\vert}\iwF{2}(t) dt \right\vert^2 \\
		&\leq 4\delta_{k'}^2 (2k') \int_{\Rz} \vert \inddif{k'}{k}(t)\vert^2 \left\vert \frac{\Mcg{c}{\denX}(t)}{\Mcg{c}{\denU}(t)}\right\vert^2 \iwF{4}(t) dt\\
		& = 4\delta_{k'}^2 (2k') \Vert\inddif{k'}{k} \Mcg{c}{\denX} / \Mcg{c}{\denU} \Vert_{\Lp[2]{}(\iwF{4})}^2 := b^2.
	\end{align*}
	Secondly, consider $v$. Using Assumption  \cref{eq::mm1} of \cref{lem::Mellinprops} we get that
	\begin{align*}
		\Ex[\denY] \left[ \left\vert Z_i \right\vert^2  \right] &\leq \Ex[\denY] \left[\left\vert \int_{\Rz} \bcut{k'}(Y_i, t) \frac{ \inddif{k'}{k}(t) \Mcg{c}{\denX}(-t)}{ \Mcg{c}{\denU}(t)}  \iwF[]{2} (t)  dt  \right\vert^2 \right] \\
		&\leq  \Ex[\denY] \left[\left\vert \int_{\Rz} Y_i^{c-1+2\pi i t} \frac{ \inddif{k'}{k}(t) \Mcg{c}{\denX}(-t)}{ \Mcg{c}{\denU}(t)}  \iwF[]{2} (t)  dt  \right\vert^2 \right] \\
		&\leq \Vert \denY  \Vert_{\Lp[\infty]{+}(\basMSy{2c-1})}  \int_{\Rz} \frac{\inddif{k'}{k}(t) \vert\Mcg{c}{\denX}(t)\vert^2}{\vert \Mcg{c}{\denU}(t)\vert^2} \iwF{4}(t) dt.
	\end{align*}
	Using  \cref{eq::norm-bound} we get 
		$\Ex[\denY] \left[ \left\vert Z_i \right\vert^2  \right]
		\leq \cstC[\denU] \cstV[\denX|\denU] \Vert\inddif{k'}{k}\Mcg{c}{\denX} / \Mcg{c}{\denU}\Vert_{\Lp[2]{}(\iwF{4})}^2 =: v$.
	Consequently, we have that
	\begin{align*}
		&n^{-1}(4v + 32b^2 (\log K)n^{-1}) (\log K) \\
		&= n^{-1}(\log K) \Vert\inddif{k'}{k}\Mcg{c}{\denX} / \Mcg{c}{\denU}\Vert_{\Lp[2]{}(\iwF{4})}^2 \left( 4 \cstC[\denU] \cstV[\denX|\denU] + 128\delta_{k'}^2 (2k')(\log K)n^{-1}\right)\\
		&= V_2^b(k').
	\end{align*}
	Consequently by \cref{lem::Bernstein-inequality} there exists a universal numerical constant $\mathcal{C}$ such that	
	\begin{align*}
		&n \iEx[\denY]{n}\left[\left( \vert \linstatb{k'}{k}\vert^2 -   V_2^b(k')\right)_+ \right]\\
		&= \iEx[\denY]{n} \left[ \left( \left\vert \sqrt{n}\linstatb{k'}{k} \right\vert^2  - (4v + 32b^2 (\log K)n^{-1}) (\log K) \right)_+ \right] \\
		& \leq \frac{8}{K} (v+16b^2n^{-1})=  \frac{\mathcal{C}}{K}\Vert\inddif{k'}{k}\Mcg{c}{\denX} / \Mcg{c}{\denU}\Vert_{\Lp[2]{}(\iwF{4})}^2 \left(\cstC[\denU] \cstV[\denX|\denU] + (2k')\delta_{k'}^2 n^{-1} \right).
	\end{align*}
	This shows the claim and concludes the proof. 
\end{proof}

\begin{lem}[Variance of the unbounded linear statistic]\label{lem::concentration-unbounded-lin}
	Consider the centered linear statistic $\linstatu{k'}{k}$ for  $k,k'\in\Nz$, $k'\geq k$  and $\delta_{k'}\in\pRz$, defined as in  \cref{eq::linear-statistic-unbouded}. Under \cref{ass:well-definedness} we have that
	\begin{align*}
		\iEx[\denY]{n} \left[\vert \linstatu{k'}{k} \vert^2 \right] \leq \frac{1}{n} \Vert (\mathds{1}_{[-k',k']} - \mathds{1}_{[-k,k]}) \Mcg{c}{\denX} / \Mcg{c}{\denU}  \Vert^2_{\Lp[2]{}(\iwF{4})} (2k') \delta_{k'}^{-2} \Ex[\denY] [Y_1^{4(c-1)}].
	\end{align*}
\end{lem}

\begin{proof}[Proof of \cref{lem::concentration-unbounded-lin}]
	We use again the notation $\inddif{k'}{k} := \mathds{1}_{[-k',k']} - \mathds{1}_{[-k,k]}$  for $k\leq k'$. By independence of $(Y_j)_{j\in\nset{n}}$ and applying Cauchy-Schwarz inequality we have that
	\begin{align*}
		n\iEx[\denY]{n}[\vert \linstatu{k'}{k}\vert^2] &= \iEx[\denY]{n}\left[\left\vert \int_{\Rz} (\ucut{k'}(Y_1, t) -  \Ex[\denY][\ucut{k'}(Y_1, t)])\frac{\inddif{k'}{k}(t) \Mcg{c}{\denX}(-t)}{ \Mcg{c}{\denU}(t)}  \iwF[]{2} (t)  dt \right\vert^2 \right]\\
		&\leq \iEx[\denY]{n}\left[\left\vert \int_{\Rz}  \inddif{k'}{k}(t) \ucut{k'}(Y_1, t) \frac{\inddif{k'}{k}(t) \Mcg{c}{\denX}(-t)}{ \Mcg{c}{\denU}(t)}  \iwF[]{2} (t)  dt \right\vert^2 \right]\\
		&\leq \left(  \int_{\Rz} \inddif{k'}{k}(t) \Ex[\denY] [\vert \ucut{k'}(Y_1,t)\vert^2] dt \right)  \left(\int_{\Rz}  \frac{\inddif{k'}{k}(t)\vert\Mcg{c}{\denX}(t)\vert^2 }{\vert\Mcg{c}{\denU}(t)\vert^2}  \iwF{4}(t) dt \right).
	\end{align*}
	Moreover, for all $t\in\Rz$ it holds
	\begin{align*}
		\Ex[\denY] [\vert \ucut{k'}(Y_1,t)\vert^2] = \Ex[\denY] [\mathds{1}_{(\delta_{k'},\infty)}(Y_1^{c-1}) Y_1^{2(c-1)}] \leq \delta_{k'}^{-2} \Ex[\denY] [Y_1^{4(c-1)}].
	\end{align*}
	Since $\int_{\Rz} \inddif{k'}{k}(t)  dt \leq 2k'$ the last bound implies the result.
\end{proof}

\bibliography{AITFMC_bib}

\end{document}